\RequirePackage{etoolbox}
\csdef{input@path}{%
 {sty/}
 {img/}
}%
\csgdef{bibdir}{bib/}

\documentclass[ba]{imsart}
\pubyear{0000}
\volume{00}
\issue{0}
\doi{0000}
\firstpage{1}
\lastpage{1}

\RequirePackage{natbib}
\usepackage{amsmath}
\usepackage{amssymb}
\usepackage{bm}
\usepackage{enumitem}
\usepackage{multirow}
\usepackage{graphicx}
\usepackage{subfigure}
\usepackage{mathrsfs}
\usepackage{amsthm}
\usepackage{lipsum}
\usepackage{amsthm}
\usepackage{amsmath}
\usepackage{natbib}
\usepackage[colorlinks,citecolor=blue,urlcolor=blue,filecolor=blue,backref=page]{hyperref}

\startlocaldefs
\numberwithin{equation}{section}
\theoremstyle{plain}
\theoremstyle{remark}
\newtheorem{theorem}{Theorem}[section]
\newtheorem{corollary}{Corollary}[section]
\newtheorem{lemma}[theorem]{Lemma}
\newtheorem{assumption}{Assumption}
\newtheorem{remark}{Remark}
\newtheorem{condition}{Condition}

\DeclareMathOperator*{\argmax}{arg\,max}
\endlocaldefs

\begin{document}
\begin{frontmatter}
\title{High-dimensional posterior consistency for hierarchical non-local priors in regression}
\runtitle{Posterior consistency for non-local priors in regression}

\begin{aug}
\author{\fnms{Xuan} \snm{Cao}\thanksref{addr1}\ead[label=e1]{xuan.cao@uc.edu}},
\author{\fnms{Kshitij} \snm{Khare}\thanksref{addr2}\ead[label=e2]{kdkhare@stat.ufl.edu}}
\and
\author{\fnms{Malay} \snm{Ghosh}\thanksref{addr2}\ead[label=e3]{ghoshm@ufl.edu}}

\runauthor{X. Cao et al.}

\address[addr1]{Department of Mathematical Sciences, University of Cincinnati,
    \printead*{e1} 
}

\address[addr2]{Department of Statistics, University of Florida,
	\printead{e2}
    \printead{e3}
}

\end{aug}

\begin{abstract}
The choice of tuning parameters in Bayesian variable selection is a critical problem in modern statistics. In particular, for Bayesian linear regression with non-local priors, the scale parameter in the non-local prior density is an important tuning parameter which reflects the dispersion of the non-local prior density around zero, and implicitly determines the size of the regression coefficients that will be shrunk to zero. Current approaches treat the scale parameter as given, and suggest choices based on prior coverage/asymptotic considerations. In this paper, we consider the fully Bayesian approach introduced in \citep{PhDthesis:Wu} with the pMOM non-local prior and an appropriate Inverse-Gamma prior on the tuning parameter to analyze the underlying theoretical property. Under standard regularity assumptions, we establish strong model selection consistency in a high-dimensional setting, where $p$ is allowed to increase at a polynomial rate with $n$ or even at a sub-exponential rate with $n$. Through simulation studies, we demonstrate that our model selection procedure can outperform other Bayesian methods which treat the scale parameter as given, and commonly used penalized likelihood methods, in a range of simulation settings.
\end{abstract}


\begin{keyword}
	\kwd{posterior consistency}
	\kwd{high-dimensional data}
	\kwd{non-local prior}
	\kwd{model selection}
	\kwd{multivariate regression}
\end{keyword}

\end{frontmatter}

\section{Introduction} \label{sec:introduction}

\noindent
The literature on Bayesian variable selection in linear regression is vast and rich. Many priors and methods have been proposed. \citet{George:McCulloch:1993} propose the stochastic search variable selection which uses the Gaussian distribution with a zero mean and a small but fixed variance as the spike prior, and another Gaussian distribution with a large variance as the slab prior. \citet*{ishwaran2005} also use Gaussian spike and slab priors, but with continuous bimodal priors for the variance of the regression coefficient to alleviate the difficulty of choosing specific prior parameters. \citet*{Narisetty:He:2014} introduce shrinking and diffusing priors as spike and slab priors, and establish model selection consistency of the approach in a high-dimensional setting. $g$-prior is introduced in \citep{Zeller:g}, and \citet{Liang:mixture:2008} further propose the mixture of $g$ priors based variable selection method and establish selection consistency. In recent years, the use of non-local priors in this context has generated a lot of interest. 

Non-local priors were first introduced by \citet{John:Rossell:non-localtesting:2010} as densities that 
are identically zero whenever a model parameter is equal to its null value in the context of hypothesis testing. Compared to 
local priors, which still preserve positive values at null parameter values, non-local prior distributions have relatively appealing 
properties for Bayesian model selection. In particular, non-local priors discard spurious covariates faster as the sample size 
$n$ grows, while preserving exponential learning rates to detect non-zero coefficients as indicated in 
\citep{John:Rossell:non-localtesting:2010}. These priors were further extended to Bayesian model selection problems in 
\citep{Johnson:Rossell:2012} by imposing non-local prior densities on a vector of regression coefficients.
Posterior distributions on the model space based on non-local priors were found to be more tightly concentrated around the maximum a posteriori (MAP) model than the posterior based on for example, $g$-priors, which tend to be more dispersed, implying that these non-local priors yield a faster rate of posterior concentration, as indicated in \citep{Shin.M:2015}.

In particular, let $\bm y_n$ denote a random vector of responses, $X_n$ an $n \times p$ design matrix of covariates, and 
$\bm \beta = (\beta_1, \beta_2, \ldots, \beta_p)$ a $p \times 1$ vector of regression coefficients. Under the linear regression 
model, 
$$
\bm y_n \sim N\left(X_n\bm \beta, \sigma^2I_n\right). 
$$
\noindent
In \citep{Johnson:Rossell:2012}, the authors introduce the product moment (pMOM) non-local prior with density
\begin{equation} \label{pmomdensity_introduction}
d_p(2\pi)^{-\frac p 2}(\tau \sigma^2)^{-rp - \frac p 2} |A_p| ^{\frac 1 2} \exp \left\{- \frac{\bm \beta_p ^ \prime A_p \bm \beta}{2 
	\tau \sigma ^2}\right\}\prod_{i =1}^{p} \beta_{i}^{2r}. 
\end{equation}
\noindent
Here $A_p$ is a $p \times p$ nonsingular matrix, $r$ is a positive integer referred to as the order of the density and 
$d_p$ is the normalizing constant independent of $\tau$ and $\sigma^2$. Variations of the density in 	
(\ref{pmomdensity_introduction}), called the piMOM and peMOM density, have also been developed in 
\citep{Johnson:Rossell:2012, RTJ:2013}. Clearly, the density in (\ref{pmomdensity_introduction}) 
is zero when any component of ${\bm \beta}$ is zero. Under appropriate regularity 
conditions, the authors in \citep{Johnson:Rossell:2012, Shin.M:2015} demonstrate that in high-dimensional settings, model 
selection procedures based on the pMOM and piMOM non-local prior densities can achieve strong model selection 
consistency, i.e, the posterior probability of the true model converges to $1$ as the sample size $n$ increases. 

As noted in \citep{Johnson:Rossell:2012}, the scale parameter $\tau$ is of particular importance, as it reflects the dispersion of 
the non-local prior density around zero, and implicitly determines the size of the regression coefficients that will be shrunk to 
zero. \citet{John:Rossell:non-localtesting:2010, Johnson:Rossell:2012}  treat $\tau$ as given and suggest 
a choice of $\tau$ which leads to a high prior probability for significant values of the regression coefficients. \citet{Shin.M:2015} again treat $\tau$ as given, and consider a setting where $p$ and $\tau$ vary with the 
sample size $n$. They show that high-dimensional model selection consistency is achieved under the peMOM prior (another variation of the priors above introduced in \citep{RTJ:2013}), as long as $\tau$ is of a larger order than $\log p$ and smaller order than $n$.  

 In the context of generalized linear model, similar to the development from $g$ prior in \citep{Zeller:g} to the mixture of $g$ prior in \citep{Liang:mixture:2008}, \citet{PhDthesis:Wu} further extends the work in \citep{Johnson:Rossell:2012,Shin.M:2015}  by proposing a fully Bayesian approach with the pMOM non-local prior and an appropriate Inverse-Gamma prior on the parameter $\tau$ referred to as the hyper-pMOM prior, following the
 nomenclature in \citep{PhDthesis:Wu}. In particular, \citet{PhDthesis:Wu} discusses the potential advantages of using hyper-pMOM priors and establish Bayes factor rates. 

The primary goal and innovation of this paper is to investigate the underlying model selection consistency for the hyper-pMOM priors in linear regression setting. The extra prior layer of prior, however, creates technical challenges for a high-dimensional theoretical consistency analysis. Under standard regularity assumptions, which include the prior over all models is restricted to ones with model size less than an appropriate function of the sample size $n$, we establish {\it posterior ratio consistency} (Theorem \ref{thm1}), i.e., the ratio of the 
maximum marginal posterior probability assigned to a ``non-true" model to the posterior probability assigned to the ``true" 
model converges to zero in probability. In particular, this implies that the true model will be the mode of the posterior 
distribution with probability tending to $1$ as $n \rightarrow \infty$. 

Next, under the additional assumption that $p$ increases 
at a polynomial rate with $n$, we show {\it strong model selection consistency} (Theorem \ref{thm2}). Strong model selection 
consistency implies that the posterior probability of the true model converges in probability to $1$ as $n \rightarrow \infty$. 
The assumption of restricting the prior over models with appropriately bounded parameter size, i.e., putting zero prior mass 
on unrealistically large models) has been used in both \citep{Narisetty:He:2014} and \citep{Shin.M:2015} for regression models. Based on reviewers' comments, we relax the polynomial rate restriction on $p$ to a sub-exponential rate by replacing the uniform type prior with a complexity prior on the model space to penalize larger models and establish model selection consistency under the complexity prior in Theorem \ref{thm4}. 
 
For the hyper-piMOM priors, \citet*{BW:2017} establish model selection consistency in the framework of generalized linear model. While there are some connections between our model and the one in \citep{BW:2017}, there are 
fundamental differences between the two models and the corresponding analyses. A detailed explanation of this is provided in Remark \ref{BW_comparison}. 

The rest of the paper is structured as follows. In Section \ref{sec:model specification} we provide our hierarchical fully 
Bayesian model. Model selection consistency results are stated in Section \ref{sec:model selection consistency}, and the 
proofs are provided in Section \ref{sec:modelselectionproofs}. Section \ref{sec:complexity} establishes the model selection consistency under the complexity prior. Details about how to approximate the posterior density for 
model selection are demonstrated in Section \ref{sec:computation}. In Section \ref{sec:experiments} and Section \ref{sec:real}, via simulation studies and real data analysis, we illustrate the model selection consistency result, and demonstrate the benefits of model selection using the fully 
Bayesian approach as compared to approaches which treat $\tau$ as given, and existing penalized likelihood approaches. We end our paper with a discussion in Section \ref{sec:discussion}.

\section{Model specification} \label{sec:model specification}

\noindent
We start by considering the standard Gaussian linear regression model with $p$ coefficients and by introducing some required 
notation. Let $\bm y_n$ denote a random vector of responses, $X_n$ an $n \times p$ design matrix of covariates, and $\bm \beta$ 
a $p \times 1$ vector of regression coefficients. Our goal is variable selection, i.e., to correctly identify all the non-zero regression  
coefficients. In light of that, we denote a model by $\bm k = \left\{k_1, k_2, \ldots, k_m\right\}$ if and only if all the non-zero 
elements of $\bm \beta$ are $\beta_{k_1}, \beta_{k_2}, \ldots, \beta_{k_m}$ and denote $\bm \beta_k = \left(\beta_{k_1}, 
\beta_{k_2}, \ldots, \beta_{k_m}\right)^T.$ For any $p \times p$ matrix $A$, let $A_k$ represent the submatrix formed from the 
columns of $A$ corresponding to model $\bm k$. In particular, Let $X_k$ denote the design matrix formed from the columns of 
$X_n$ corresponding to model $\bm k$. For the rest of the paper, simply let $k = |\bm k|$ represent the cardinality of model ${\bm k}$ for notational 
convenience. 

The class of pMOM densities (\ref{pmomdensity_introduction}) can be used for model selection through the following hierarchical 
model.
\begin{align} \label{modelspecification}
  &\bm{Y}_n \mid \bm \beta_k, \sigma^2, \bm k \sim N(X_k \bm \beta_k, \sigma^2 I_n),\\
 &\pi\left(\bm \beta_k \mid \tau, \sigma^2, \bm k\right) = d_k(2\pi)^{-\frac k 2}(\tau \sigma^2)^{-rk - \frac k 2} |A_k| ^{\frac 1 2} \exp \left\{- \frac{\bm \beta_k ^ \prime A_k \bm \beta_k}{2 \tau \sigma ^2}\right\}\prod_{i =1}^{k} \beta_{k_i}^{2r}, \label{model:pmom}\\
  &\pi(\tau) = \frac{\left(\frac n 2\right)^{\frac 1 2}}{\Gamma(\frac 1 2)} \tau^{- \frac 3 2} e^{-\frac{n}{2\tau}}, \label{model:tau}\\
  &\pi\left(\sigma^2\right) = \frac{\left( \alpha_2 \right)^{\alpha_1}}{\Gamma(\alpha_1)} \left(\sigma^2\right)^{-(\alpha_1 + 1)} e^{-\frac{\alpha_2}{\sigma^2}}. \label{model:5}
\end{align}
Note that in the currently presented hierarchical model, no specific form/condition has yet been assigned to the prior over the space of models. 
Some standard regularity assumptions for this prior will be provided later in Section \ref{sec:model selection consistency}. 
\begin{figure}[h] 
	\centering
	\includegraphics[width=95mm,height=50 mm]{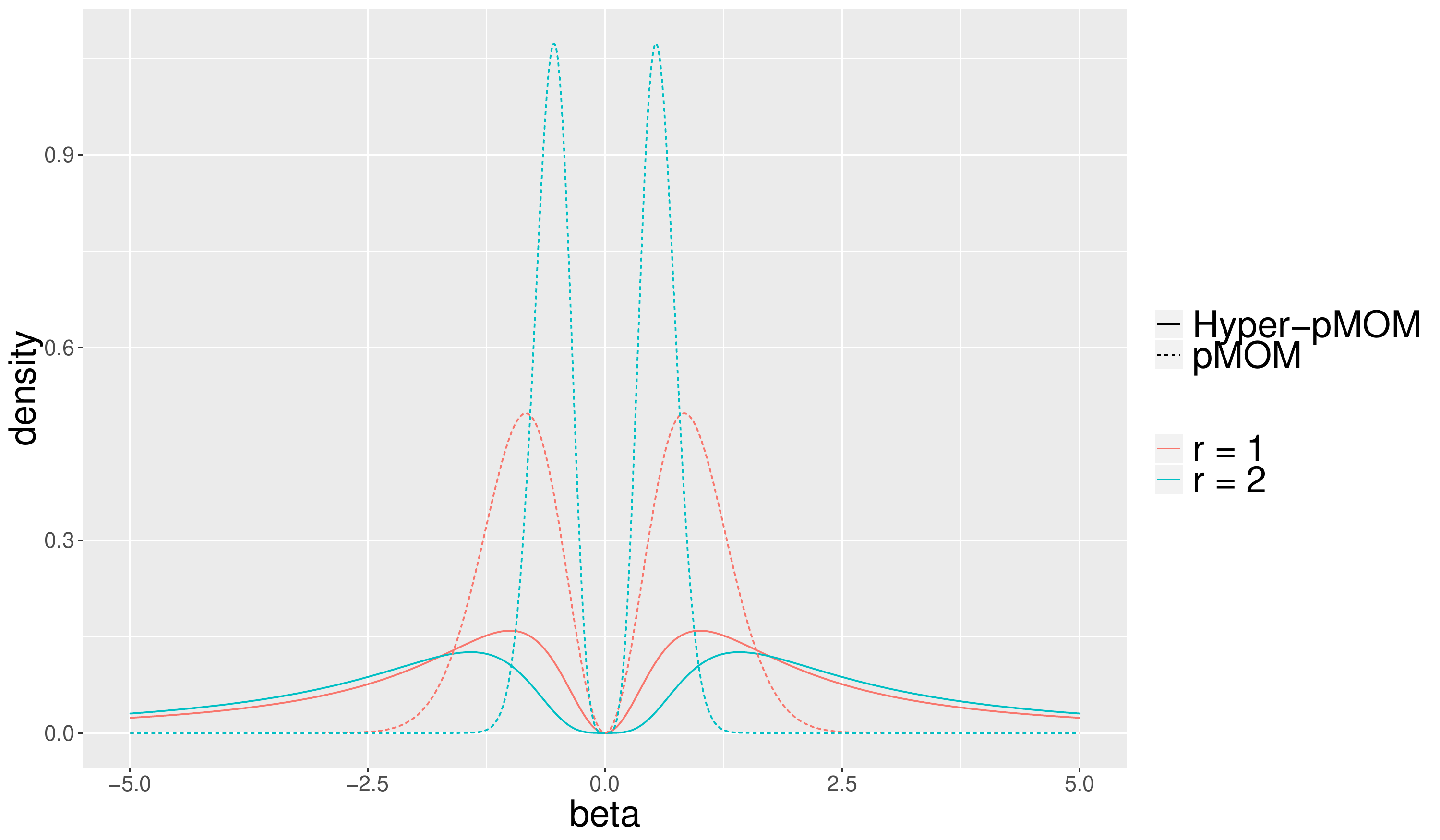}
	\caption{Comparison: Hyper-pMOM and pMOM when $p = 1$.}
	\label{fig:one_dimension}
\end{figure} 
\begin{figure} [h]
	\centering
	\begin{subfigure}[pMOM]
		{\includegraphics[width=45mm]{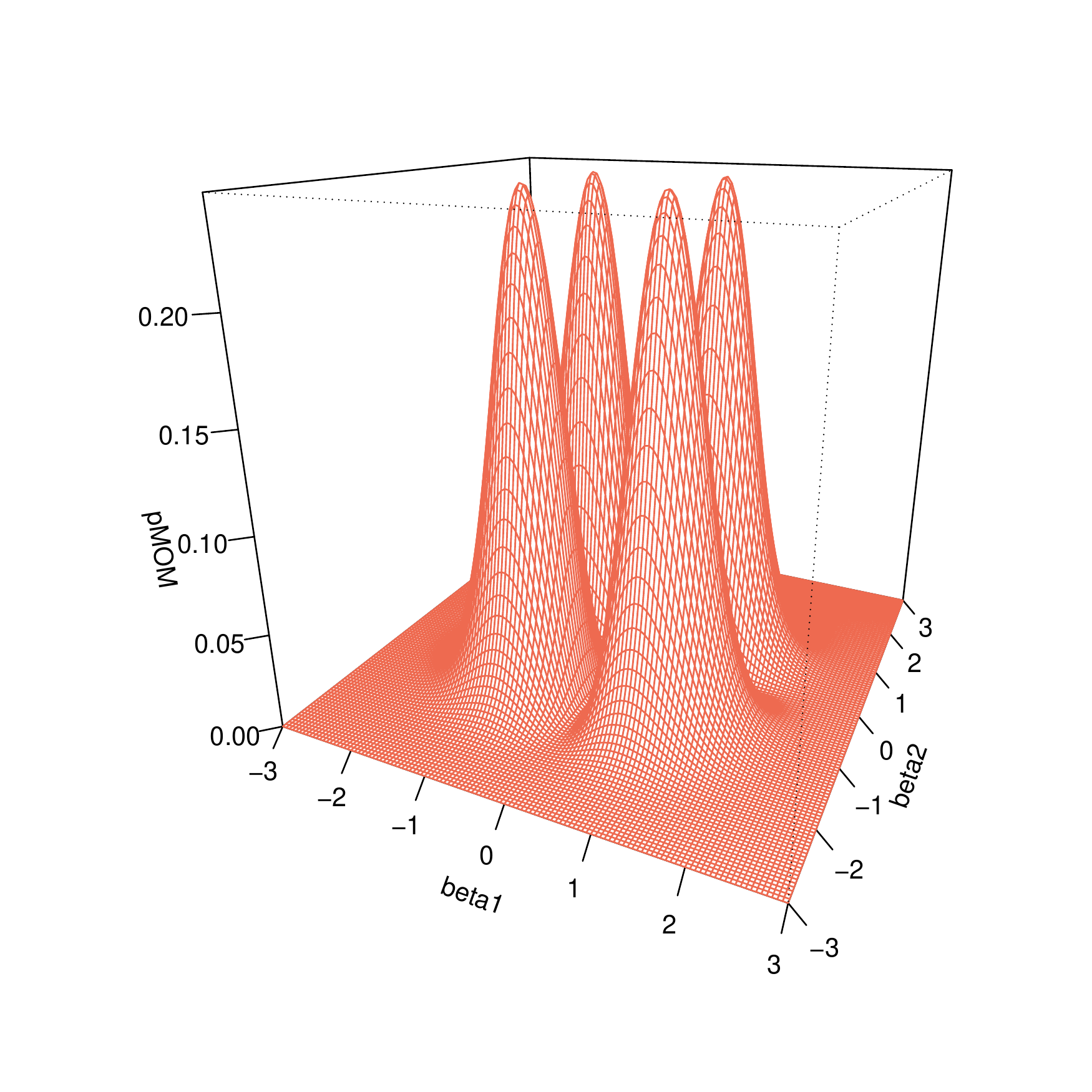}}
	\end{subfigure}
	\begin{subfigure}[Hyper-pMOM]
		{\includegraphics[width=45mm]{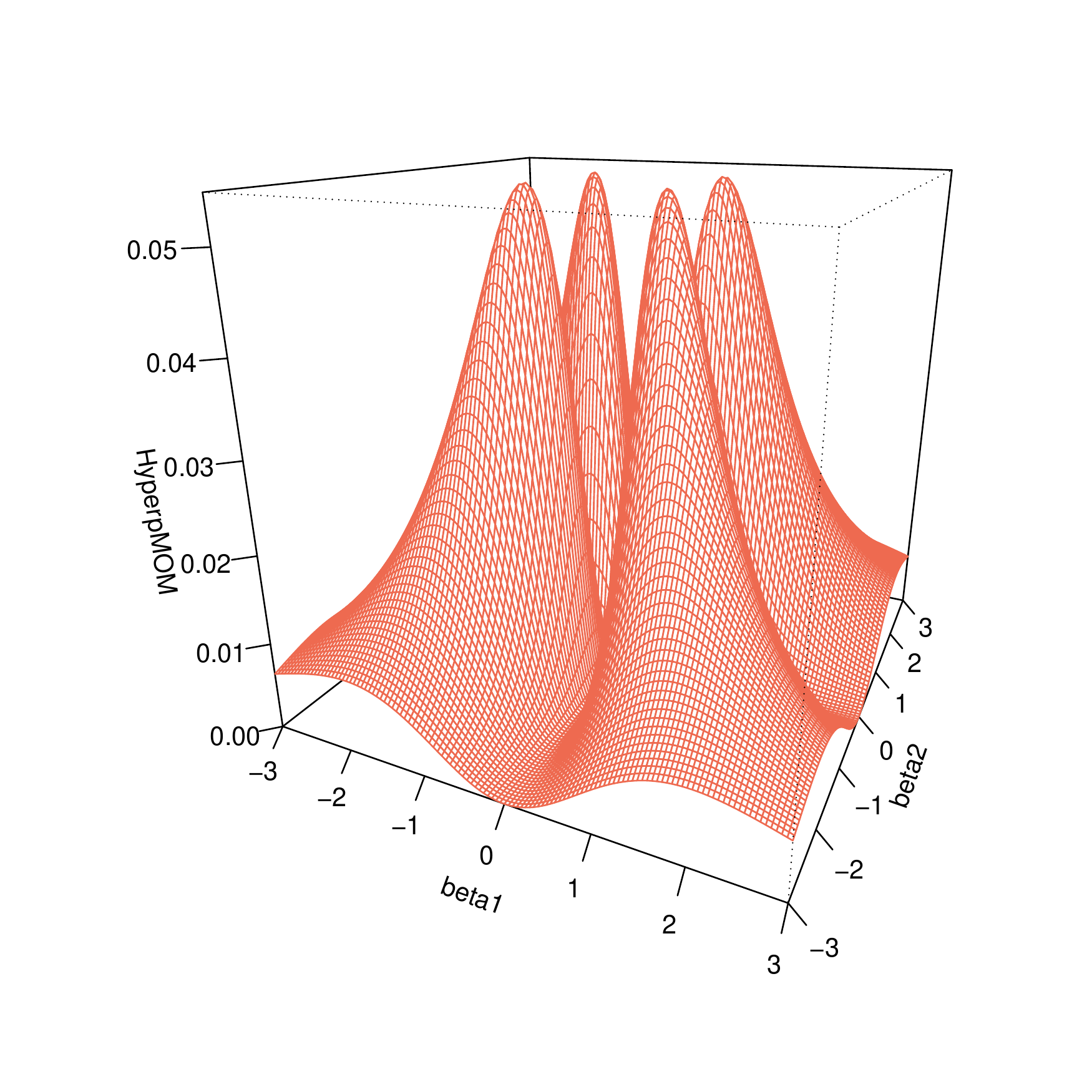}}
	\end{subfigure}	
	\caption{Comparison: Hyper-pMOM and pMOM when $p = 2$.}
	\label{fig:two_dimension}
\end{figure} 
Following the nomenclature in \citep{PhDthesis:Wu}, we refer to the mixture of priors in (\ref{model:pmom}) and (\ref{model:tau}) as the hyper-pMOM prior. In particular, one can show the implied marginal density of $\bm \beta_k$ after integrating out $\tau$ have the following expression
\begin{equation} \label{marginal_beta}
\pi\left(\bm \beta_k \mid \sigma^2, \bm k\right) = \frac{\left(\frac n 2\right)^{\frac 1 2}}{\Gamma(\frac 1 2)} \frac{\Gamma(rk+\frac k 2 +\frac 1 2)}{(\frac{n}{2} + \frac{\bm \beta_k ^ \prime A_k \bm \beta_k}{2\sigma^2})^{rk + \frac k 2 + \frac 1 2}}d_k(2\pi)^{-\frac k 2}\sigma^{-2rk -  k} |A_k| ^{\frac 1 2} \prod_{i =1}^{k} \beta_{k_i}^{2r}.
\end{equation}
Note that compared to the pMOM density in (\ref{pmomdensity_introduction}) with given $\tau$, $\pi\left(\bm \beta_k \mid \sigma^2, \bm k\right)$ now possesses thicker tails, which induces prior dependence. See Figure \ref{fig:one_dimension} and Figure \ref{fig:two_dimension}, where we plot the marginal density $\pi\left(\bm \beta_k \mid \sigma^2, \bm k\right)$ when $A_p = 1$, $\sigma^2 = 1$ and $n = 1$ for the univariate and bivariate case, respectively. In addition, the hyper-pMOM prior could achieve better model selection performance especially for small samples.  See for example \citep{Liang:mixture:2008} that investigates the finite sample performance for hyper-$g$ priors. 

By (\ref{modelspecification}) and Bayes' rule, the resulting posterior probability for model $\bm k$ is denoted by,
\begin{align} \label{model posterior}
\pi(\bm k | \bm y_n) = \frac{\pi(\bm k)}{\pi(\bm y_n)}m_{\bm k}(\bm y_n),
\end{align}
where $\pi(\bm y_n)$ is the marginal density of $\bm y_n$, and $m_{\bm k}(\bm y_n)$ is the marginal density of $\bm y_n$ under 
model $\bm k$ given by,
\begin{align} \label{marginal density}
&m_{\bm k}(\bm y_n) \nonumber\\
=&\int_{0}^{\infty}\int_{0}^{\infty}\pi\left(\bm{y}_n \mid \bm \beta_k, \sigma^2, \bm k\right)\pi\left(\bm \beta_k \mid \tau, \sigma^2, 
\bm k\right)\pi(\tau)\pi\left(\sigma^2\right)d\bm \beta_k d\sigma^2 d\tau \nonumber\\
=& \frac{\left(\frac n 2\right)^{\frac 1 2}}{\Gamma(\frac 1 2)} \frac{\left( \alpha_2 \right)^{\alpha_1}}{\Gamma(\alpha_1)}\int_{0}^{\infty}
\int_{0}^{\infty} d_k(2\pi)^{-\frac k 2}(\tau \sigma^2)^{-rk - \frac k 2} |A_k| ^{\frac 1 2} \exp \left[- \frac{\bm \beta_k ^ \prime A_k \bm 
	\beta_k}{2 \tau \sigma ^2}\right]\prod_{i =1}^{k} \beta_{k_i}^{2r} \nonumber\\
&\times \frac{1}{\left(2\pi\sigma^2\right)^{\frac n 2}}\exp\left\{-\frac{(\bm y_n - X_k\bm \beta_k)^T(\bm y_n - X_k\bm \beta_k)}
{2\sigma^2}\right\}  \tau^{- \frac 3 2} e^{-\frac{n}{2\tau}} \left(\sigma^2\right)^{-(\alpha_1 + 1)} e^{-\frac{\alpha_2}{\sigma^2}} d\bm 
\beta_k d\sigma^2 d\tau \nonumber\\
=&  d_k\frac{\frac{\left(\frac n 2\right)^{\frac 1 2}}{\Gamma(\frac 1 2)}}{(\sqrt{2\pi})^n}\frac{\left( \alpha_2 \right)^{\alpha_1}}
{\Gamma(\alpha_1)} |A_k|^{\frac 12} \nonumber\\
&\times \int_{0}^{\infty}\int_{0}^{\infty} (\sigma^2)^{-\left(\frac n 2 + rk + \alpha_1 + 1\right)} \exp\left\{-\frac{R_k + 2\alpha_2}
{2\sigma^2}\right\} \tau^{-rk - \frac k 2 - \frac 3 2}e^{-\frac{n}{2\tau}}\frac{E_k(\prod_{i=1}^{k}\beta_{k_i}^{2r})}{|C_k|^{\frac 1 2}}
d\sigma^2 d\tau,
\end{align} 

\noindent
where $C_k = X_k^TX_k + \frac {A_k}{\tau}, R_{\bm k} =\bm y_n^T(I_n - X_k C_k^{-1}X_k^T)\bm y_n,$ and $E_k(.)$ denotes the 
expectation with respect to a multivariate normal distribution with mean $\tilde{\bm \beta_k} = C_k^{-1}X_k^T\bm y_n$, and 
covariance matrix $V = \sigma^2C_k^{-1}$. In particular, these posterior probabilities can be used to select a model by computing 
the posterior mode defined by
\begin{equation} \label{a4}
\hat{\bm k} =  \argmax_{\bm k} \pi({\bm k}|\bm y_n).
\end{equation}

\section{Model selection consistency: main results} \label{sec:model selection consistency}

\noindent
In this section we will explore the high-dimensional asymptotic properties of the Bayesian model selection approach specified in 
Section \ref{sec:model specification}. In particular, we consider a setting where the number of regression coefficients $p = p_n$ 
increases with the sample size $n$. The true data generating mechanism is given by $Y_n = X_n \bm \beta_0 + 
\boldsymbol{\epsilon}_n$. Here $\bm \beta_0$ is the true $p_n$ dimensional vector of regression coefficients, whose dependence 
on $n$ is suppressed for notational convenience, and the entries of $\boldsymbol{\epsilon}_n$ are i.i.d Gaussian with mean zero 
and variance $\sigma_0^2$. As in \citep{Johnson:Rossell:2012}, we assume that the true vector of regression coefficients is sparse, i.e., all the entries of 
$\bm \beta_0$ are zero except those corresponding to a subset ${\bm t} \subseteq \{1,2, \ldots, p_n\}$, and ${\bm t}, 
\bm \beta_{0,t}, \sigma_0^2$ do not vary with $n$. Our results can be easily extended to the case where $|{\bm t}|$, and entries of 
$\bm \beta_{0,t}$ and $\sigma_0^2$ vary with $n$ but stay bounded. However, we assume these quantities stay fixed for ease 
of exposition. 

For any $p \times p$ symmetric matrix $A$, let $eig_1(A) 
\le eig_2(A) \ldots \leq eig_p(A)$ be the ordered eigenvalues of $A$ and denote the $j$-th largest nonzero eigenvalue as $\nu_j(A)$. Let $\lambda_k^m = \min_{1 \le j \le \min(n,k)} \nu_j\left(\frac{X_k^TX_k}{n}\right)$ and $\lambda_k^M = \max_{1 \le j \le \min(n,k)} \nu_j\left(\frac{X_k^TX_k}{n}\right)$, respectively. In order to establish our asymptotic results, we need the following mild regularity assumptions.
\begin{assumption}\label{assumptiontruemodel}
	There exist $\epsilon_n < 1$, such that $0< \epsilon_n \le \lambda_k^m \le \lambda_k^M \le \epsilon_n^{-1}$, where $\epsilon_n^{-1} = O\left((\log n)^{\frac 1 8}\right)$.
\end{assumption}
\begin{assumption}\label{assumptionpvalue}
	$p = O\left(n^{\gamma}\right),$ where $\gamma < r.$ 
\end{assumption}
\begin{assumption}\label{assumptionmodelsize}
	$\pi(\bm k) = 0$ for all $|\bm k| > q_n,$ where $q_n = O\left(n^\xi\right)$ and $\xi < 1$. 
\end{assumption}
\begin{assumption}\label{assumptionprior}
	There exists a constant $\omega > 0$, such that $\frac{\pi(\bm t)}{\pi(\bm k)} > \omega$ for every ${\bm k}$ with 
	$\pi({\bm k}) > 0$. 
\end{assumption}
\begin{assumption}\label{assumptionhyper}
	For every $n \ge 1$, the hyper-parameter for the non-local pMOM prior in \ref{modelspecification} satisfy $0 < a_1 < eig_1(A_p) \le eig_2(A_p) \le \ldots \le eig_p(A_p) < a_2 < \infty$. Here $a_1, a_2$ are constants not depending on $n$.
\end{assumption}

\noindent
\citet{Johnson:Rossell:2012} assume all the eigenvalues of $\frac{X^TX}{n}$ to be bounded by a constant, which is unrealistic to achieve  in the high-dimensional setting. In our work, Assumption \ref{assumptiontruemodel} assumes that non-zero eigenvalues of any sub-matrices of the design matrix not to be bounded by a constant, but to be uniformly bounded over a function of $n$. Assumption \ref{assumptionhyper} is standard which assumes the prior covariance matrix are uniformly bounded in $n$. Note that for the default value of $A_p = I_p$, Assumption \ref{assumptionhyper} is immediately satisfied. Assumption \ref{assumptionmodelsize} states that the prior on the space of the $2^{p_n}$ possible models, places zero mass on unrealistically large models (identical to Assumption in \citep{Shin.M:2015}). 
Assumption \ref{assumptionprior} states that the ratio of the prior probabilities assigned to the true model and any non-true 
model stays bounded below in $n$ (identical to Assumption in \citep{Johnson:Rossell:2012}). This type of priors have also been considered in \citep{Liang:2015} and \citep{Shin.M:2015}. Assumption \ref{assumptionpvalue} states that 
$p$ can grow at an appropriate polynomial rate with $n$. In Section \ref{sec:complexity}, we also give the consistency results under the complexity priors on the model space, which penalize larger models, and consequently relax the assumption on the rate at which $p$ can be growing.

We now state and prove the main model selection consistency results. Our first result establishes what we refer to as posterior 
ratio consistency. This notion of consistency implies that the true model will be the mode of the posterior distribution among all the 
models with probability tending to $1$ as $n \rightarrow \infty.$
\begin{theorem}[Posterior ratio consistency for hyper-pMOM priors] \label{thm1}
	Under Assumptions \ref{assumptiontruemodel}, \ref{assumptionmodelsize}, \ref{assumptionprior} and \ref{assumptionhyper}, for the hierarchical model in (\ref{modelspecification}) to (\ref{model:5}) with hyper-pMOM priors, the following holds: 
	$$
	\max_{\bm k \ne {\bm t}} \frac{\pi(\bm k|\bm y_n)}{\pi(\bm t|\bm y_n)} \rightarrow 0, \quad \mbox{as } n \rightarrow \infty. 
	$$
	
	\noindent
	In particular, it implies that the probability that the posterior mode $\hat{\bm k}$ defined in (\ref{a4}) is equal to 
	the true model $\bm t$ will converge to $1$, i.e., 
	$$
	P(\bm t = \argmax_{\bm k} \pi(\bm k|\bm y_n)) \rightarrow 
	1, \quad \mbox{as } n \rightarrow \infty. 
	$$ 	
\end{theorem}

\noindent
We would like to point out that posterior ratio consistency 
(Theorems \ref{thm1}) does not require any restriction on the number of predictors. This requirement is only needed for strong selection consistency (Theorem \ref{thm2}). Next, we establish a stronger result which implies that the posterior mass assigned to the true model $\bm t$ converges to $1$ in 
probability. We refer to this notion of consistency as strong selection consistency. 
\begin{theorem}[Strong selection consistency for hyper-pMOM priors] \label{thm2}
	Under Assumptions 1-5, with $\xi < 1 - \frac {4\gamma} {3r}$ in Assumption \ref{assumptionmodelsize}, for the hierarchical model in (\ref{modelspecification}) to (\ref{model:5}) with hyper-pMOM priors, the following holds:
	$$
	\pi(\bm t | \bm y_n) \rightarrow 1, \quad \mbox{as } n \rightarrow 
	\infty. 
	$$
\end{theorem}  

\noindent
The above results have been established under the pMOM priors. Another class of non-local priors introduced in \citep{Johnson:Rossell:2012} are the 
piMOM priors on the regression coefficients, for which the density of the regression coefficients under the model 
${\bm k} = \{k_1, k_2, \ldots, k_m\}$ is given by 
\begin{equation} \label{pimomdensity_introduction}
\frac{(\tau\sigma^2)^{\frac{r|{\bm k}|}{2}}}{\Gamma(\frac r 2)^{|{\bm k}|}} \prod_{i=1}^{|{\bm k}|} |\beta_{k_i}|^{-(r+1)} 
\exp\left(-\frac{\tau\sigma^2}{\beta_{k_i}^2}\right), 
\end{equation}

\noindent
where $r$ is a positive integer and is refereed to as the order of the density. The corollary below establishes strong model selection 
consistency under the hyer-piMOM priors with piMOM priors on each linear regression coefficient (conditional on the hyper parameter $\tau$) and an Inverse-Gamma prior on $\tau$. The consistency can be obtained immediately by combining Theorem \ref{thm2} with Eq. (59) and (60) in the supplementary material for \citep{Johnson:Rossell:2012}. 
\begin{corollary}[Strong selection consistency for hyper-piMOM priors] \label{corollary1}
	Under the same conditions as in Theorem \ref{thm2}, when piMOM priors are imposed on $\bm \beta_k$ in model (\ref{model:pmom}), the following holds:
	$$
	\pi(\bm t | \bm y_n) \rightarrow 1, \mbox{ as } n \rightarrow 
	\infty. 
	$$
\end{corollary}

\begin{remark} \label{BW_comparison}
	In the context of generalized linear regression, \citet{BW:2017} consider 
	the hierarchical Bayesian model with the following hyer-piMOM priors on regression coefficients.
	\begin{eqnarray*}
		& & \bm \beta_k \mid \tau_i \sim \frac{(\tau\sigma^2)^{\frac{r|{\bm k}|}{2}}}{\Gamma(\frac r 2)^{|{\bm k}|}} \prod_{i=1}^{|{\bm k}|} |\beta_{k_i}|^{-(r+1)} \exp\left(-\frac{\tau_i\sigma^2}{\beta_{k_i}^2}\right)\\ 
		& & \tau_i \sim \mbox{Inverse-Gamma }(\frac{(r + 1)}2 , \lambda). 
	\end{eqnarray*}
	\noindent
	In particular, the authors put an independent piMOM prior on each linear regression coefficient (conditional on the hyper parameter $\tau_i$), and an Inverse-Gamma prior on $\tau_i$. In this setting, \citet{BW:2017} establish strong selection consistency for the regression coefficients (assuming the prior is constrained to leave out unrealistically large models). There are similarities between the models and the consistency analysis in \citep{BW:2017} and our work as in the usage of non-local priors and Inverse-Gamma distribution. However, despite these similarities, there are some fundamental differences in the two models and the corresponding analysis. Firstly, unlike the piMOM prior, the pMOM prior in our model does not in general correspond to assigning an independent prior to each entry of $\bm \beta_k$. In particular, pMOM distributions introduce correlations among the entries in $\bm \beta_k$ and creates more theoretical challenges. Because of the correlation introduced, some properties like detecting small coefficients are not apparent in our case. Also, the pMOM prior imposes exact sparsity in $\bm \beta_k$, which is not the case in \citep{BW:2017}.  Hence it is structurally different than the prior in \citep{BW:2017}. Secondly, in order to prove consistency results, \citet{BW:2017} assume the product of the response variables and the entries of design matrix are bounded by a constant. The former assumption is rarely seen in the literatures and the latter assumption can be problematic in practice. See Assumption C1 in \citep{BW:2017}. In addition, Assumption C2 in \citep{BW:2017} states the lower bound for the true regression coefficients, which is not required in our analysis. Thirdly, in terms of proving posterior consistency, we bound the ratio of posterior probabilities for a non-true model and the true model by a `prior term' which results from the  Inverse-Gamma prior on $\tau$, and a `data term'. The consistency proof is then a careful exercise in balancing these two terms against each other on a case-by-case basis, while \citet{BW:2017} directly follow the proof in \citep{Shin.M:2015} and requires additional assumptions on the Hessian matrix. 
	
\end{remark}

\section{Proof of Theorems \ref{thm1} and \ref{thm2}} \label{sec:modelselectionproofs}

\noindent
The proof of Theorems \ref{thm1} and \ref{thm2} will be broken up into several steps. First we denote for any model $\bm k$,  
$R_k^* = \bm y_n^T\left(I - X_k(X_k^TX_k)^{-1}X_k^T\right)\bm y_n,$ and $P_k = X_k(X_k^TX_k)^{-1}X_k^T$. Our method of 
proving consistency involves approximating $R_t$ and $R_k$ (in (\ref{marginal density})) with $R_t^*$ and $R_k^*$ 
respectively. Fix a model ${\bm k} \neq {\bm t}$ arbitrarily, and let $\bm u = \bm k \cup \bm t$ and $u = |\bm u|$ be the cardinality of $\bm u$. Note that $\frac{R_t^*}{\sigma_0^2} 
\sim \chi^2_{n-t}$, $\frac{R_u^*}{\sigma_0^2} \sim \chi^2_{n-u}$, $R_{u \cap t^c}^* \sim \chi_{u-t}^2$, and 
$\frac{\bm y_n^TP_u\bm y_n}{\sigma_0^2} \sim \chi^2_u\left(\frac{\bm \beta_0^TX_t^TX_t\bm \beta_0}{\sigma_0^2}\right)$. Next, we 
establish two tail probability bounds for the $\chi^2$ distribution and the non-central $\chi^2$ distribution respectively, which will be useful in our analysis. 
\begin{lemma} \label{chisquaredtail}
	For any $a > 0$, we have the following two inequalities,
	\begin{align}
	P\left(\lvert\chi_p^2 - p\rvert > a\right) &\le 2\exp\left(-\frac{a^2}{4p}\right),\\
	P\left(\chi_p^2(\lambda) - (p + \lambda) > a\right) &\le \exp\left(-\frac p 2\left\{\frac a {p+\lambda} - \log\left(1+\frac a {p+\lambda}\right)\right\}\right).
	\end{align}
\end{lemma}
\noindent
The proof for Lemma \ref{chisquaredtail} is provided in the supplemental document.
The following result is immediate from Lemma \ref{chisquaredtail}.
\begin{align} \label{chisquaredRt}
\begin{split}
P\left[\left\lvert\frac{R_t^*}{\sigma_0^2} - (n-t)\right\rvert > \sqrt{n-t}\log n \right] &\le P\left[\left\lvert\frac{R_t^*}{\sigma_0^2} - (n-t)\right\rvert > 4\sqrt{(n-t)\log n} \right] \\
&\le 2n^{-1} \rightarrow 0,
\end{split}
\end{align}
as $n \rightarrow \infty.$
Similarly, we have
\begin{align} \label{chisquaredRk}
\begin{split}
P\left[\left\lvert\frac{R_u^*}{\sigma_0^2} - (n-u)\right\rvert > \sqrt{n-u}\log n \right] \le 2n^{-1} \rightarrow 0,
\end{split}
\end{align}
and
\begin{align} \label{chisquaredRk-t}
\begin{split}
P\left[\left\lvert\frac{R_{u \cap t^c}^*}{\sigma_0^2} - (u-t)\right\rvert > \sqrt{u-t}\log n \right]  \le 2n^{-1} \rightarrow 0,
\end{split}
\end{align}
as $n \rightarrow \infty$. Next, by Lemma \ref{chisquaredtail}, since $\bm u 
\supset \bm t$, we have 
\begin{align} \label{chisquaredPk}
\begin{split}
&P\left[\frac{\bm y_n^TP_u\bm y_n}{\sigma_0^2} - \left(u + \frac1{\sigma_0^2}\bm \beta_0^TX_t^TX_t\bm \beta_0\right) > n\log n - u - \frac1{\sigma_0^2}\bm \beta_0^TX_t^TX_t\bm \beta_0 \right]\\
\le& \exp\left\{-\frac u 2 \left\{\frac{n\log n}{u + \frac1{\sigma_0^2}\bm \beta_0^TX_t^TX_t\bm \beta_0} - \log\left(1 + \frac{n\log n}{u + \frac1{\sigma_0^2}\bm \beta_0^TX_t^TX_t\bm \beta_0}\right)\right\}\right\} \\
\le& \exp\left\{-\frac u 4 \left\{\frac{\log n}{1 + \frac 1{\sigma_0^2\epsilon_n}\bm \beta_0^T\bm \beta_0} \right\}\right\}.\\
\preceq & n^{-c^\prime u} \rightarrow 0,
\end{split}
\end{align}

\noindent
as $n \rightarrow \infty$, for some constant $c^\prime > 0$. Define the event $C_n$ as 
\begin{align} \label{largeprobset}
\begin{split}
C_n =& \left\{\left\lvert\frac{R_t^*}{\sigma_0^2} - (n-t)\right\rvert > \sqrt{n-t}\log n \right\} \cup \left\{\left\lvert\frac{R_u^*}{\sigma_0^2} - (n-u)\right\rvert > \sqrt{n-u}\log n\right\} \\
&\cup \left\{\left\lvert\frac{R_{u \cap t^c}^*}{\sigma_0^2} - (u-t)\right\rvert > \sqrt{u-t}\log n\right\} \cup \left\{\frac{\bm y_n^TP_u\bm y_n}{\sigma_0^2} > n\log n\right\}, 
\end{split}
\end{align}

\noindent
It follows from (\ref{chisquaredRt}), (\ref{chisquaredRk}), (\ref{chisquaredRk-t}), and (\ref{chisquaredPk}), that $P(C_n) 
\rightarrow 0$ as $n \rightarrow \infty$. 

We now analyze the behavior of the posterior ratio under different 
scenarios in a sequence of lemmas. Recall that our 
goal is to find an upper bound for the posterior ratio, such that the upper bound converges to $0$ as $n \rightarrow \infty$. {\bf For the following lemmas, we will restrict ourselves to the event 
	$C_n^c$}. The following lemma establishes the upper bound of the marginal posterior ratio for any ``non-true" model compared to the true model.
\begin{lemma} \label{lm1}
	Under Assumption \ref{assumptiontruemodel} and Assumption \ref{assumptionhyper}, for all  $\bm k \neq \bm t,$ there exists $N$ (not depending on $k$), such that when $n > N$,
	\begin{align}
	\begin{split}
	\frac{m_{\bm k}(\bm y_n)}{m_{\bm t}(\bm y_n)} <& BA^k\left(\frac V{\epsilon_n^2}\right)^{rk}k^{k}n^{-(k-t)}\frac{\{R_t^* + 2\alpha_2\}^{\frac n 2 + rt + \alpha_1}}{\{R_k^* + 2\alpha_2\}^{\frac n 2 + rk + \alpha_1}}\\
	&+ BA^kk^{(r+1)k}n^{-(r+1)(k-t)-\frac 3 4 rk -rt}\frac{\{R_t^* + 2\alpha_2\}^{\frac n 2 + rt + \alpha_1}}{\{R_k^* + 2\alpha_2\}^{\frac n 2 + \alpha_1}},
	\end{split}
	\end{align}
	where $A, B$ are constants and $V = \epsilon_n^{-4}\hat \beta_{u}^T \hat \beta_{u}$, in which $\hat \beta_{u}^T = (X_u^TX_u)^{-1}X_u\bm y_n$ with $\bm u = \bm k \cup \bm t$.
\end{lemma}
\noindent
The proof for Lemma \ref{lm1} is provided in supplemental document. 
The next two lemmas provide the upper bound of the marginal posterior ratio for $\bm y_n$ under different cases of the ``non-true" model $\bm k$ with proof provided in the supplemental document. 
\begin{lemma} \label{lm3}
	Under Assumptions \ref{assumptiontruemodel}, \ref{assumptionmodelsize} and \ref{assumptionhyper}, for all $\bm k \nsupseteq \bm t$, there exists $N$, such that when $n > N^\prime$ (not depending on $k$),
	\begin{align}
	\frac{m_{\bm k}(\bm y_n)}{m_{\bm t}(\bm y_n)} < K^{\prime}(L^{\prime})^kn^{-\frac 3 4 rk},
	\end{align}	
	where $K^{\prime}$ and $L^{\prime}$ are constants.
\end{lemma}
\begin{lemma} \label{lm5}
	Under Assumptions \ref{assumptiontruemodel}, \ref{assumptionmodelsize} and \ref{assumptionhyper}, for all $\bm k \supset \bm t$, there exists $N^{\prime\prime}$ (not depending on $k$), such that when $n > N^{\prime\prime}$,
	\begin{align}
	\frac{m_{\bm k}(\bm y_n)}{m_{\bm t}(\bm y_n)} < S^{\prime}(T^{\prime})^{k-t}n^{-\min\left\{\frac 3 4, 1-\xi \right\} r(k-t)},
	\end{align}
	where $S^{\prime}$ and $T^{\prime}$ are constants.
\end{lemma}
\begin{proof}[Proof of Theorem \ref{thm1} and \ref{thm2}]
The proof of Theorem \ref{thm1} will follow immediately from these two lemmas. By Lemma \ref{lm3}, if we restrict to $C_n^c,$ for any $\bm k \neq \bm t,$ if $\bm k \nsupseteq \bm t,$
\begin{align*}
\frac{m_{\bm k}(\bm y_n)}{m_{\bm t}(\bm y_n)} < K^{\prime}(L^{\prime})^kn^{-\frac 3 4 rk}  \rightarrow 0, \mbox{ as } n \rightarrow \infty.
\end{align*}
Otherwise, when $\bm k \supset \bm t,$
\begin{align*}
\frac{m_{\bm k}(\bm y_n)}{m_{\bm t}(\bm y_n)} < S^{\prime}(T^{\prime})^{k-t}n^{--\min\left\{\frac 3 4, 1-\xi\right\} r(k-t)} \rightarrow 0, \mbox{ as } n \rightarrow \infty.
\end{align*}
Note that $P\left(C_n^c\right) \rightarrow 1$ as $n \rightarrow \infty$. Following from (\ref{model posterior}) and Assumption \ref{assumptionprior}, when $\bm k \neq \bm t$, we have 
\begin{align} \label{thm1proof}
\frac{\pi(\bm k|\bm y_n)}{\pi(\bm t|\bm y_n)} \le \frac 1 \omega \frac{m_{\bm k}(\bm y_n)}{m_{\bm t}(\bm y_n)} \rightarrow 0, \mbox{ as } n \rightarrow \infty.
\end{align}

We now move on to the proof of Theorem \ref{thm2}. First note that when $\xi < 1 - \frac {4\gamma} {3r},$ we have 
\begin{equation} \label{thm3proof}
\min\left\{\frac 3 4,1-\xi\right\}r > \frac{3}{4}(1-\xi)r >\gamma.
\end{equation}
It follows from (\ref{thm1proof}) and Assumption \ref{assumptionpvalue} that if we restrict to $C_n^c$, then 
\begin{align*}
\frac{1-\pi(\bm t|\bm y_n)}{\pi(\bm t|\bm y_n)} =& \sum_{\bm k \neq \bm t}\frac{\pi(\bm k)m_{\bm k}(\bm y_n)}{\pi(\bm t)m_{\bm t}(\bm y_n)}\\
\le & \frac 1 \omega\sum_{\bm k \nsupseteq \bm t}\frac{m_{\bm k}(\bm y_n)}{m_{\bm t}(\bm y_n)} + \frac 1 \omega\sum_{\bm k \supset \bm t}\frac{m_{\bm k}(\bm y_n)}{m_{\bm t}(\bm y_n)}\\
\le & \frac 1 \omega\sum_{k = 1}^{q_n} \binom pk K^\prime(L^\prime)^kp^{-\frac{ \frac 3 4 r}{\gamma} k} \\
&+  \frac 1 \omega\sum_{k-t = 1}^{q_n-t} \binom {p-t}{k-t} S^{\prime}(T^{\prime})^{k-t} p^{-\frac{\min\left\{\frac 3 4, 1-\xi \right\}r}{\gamma}(k-t)}.
\end{align*}
Using $\binom {p} {k} \le p^{k}$ and (\ref{thm3proof}), we get
\begin{align*}
\frac{1-\pi(\bm t|\bm y_n)}{\pi(\bm t|\bm y_n)} \rightarrow 0, \mbox{ i.e. } \pi(\bm t|\bm y_n) \rightarrow 1,
\end{align*}
as $n \rightarrow \infty.$
\end{proof}
\section{Results for Complexity Priors} \label{sec:complexity}
\noindent
Note that under our model prior specified in Assumption \ref{assumptionprior}, to achieve strong selection consistency, we are restricting $p$ to grow at a polynomial rate of $n$ (see Assumption \ref{assumptionpvalue}). To address this limitation, based on reviewers' comments, we investigate the theoretical property under the complexity priors introduced in \citep{Castillo:2015}. The hierarchical model with complexity priors placed on the model space, adapted to our notation and framework, can be
	described as follows: 
	\begin{eqnarray} \label{complexity}
	\begin{split}
	&  \bm{Y}_n \mid \bm \beta_k, \sigma^2, \bm k \sim N(X_k \bm \beta_k, \sigma^2 I_n)\\
	&  \pi\left(\bm \beta_k \mid \tau, \sigma^2, \bm k\right) = d_k(2\pi)^{-\frac k 2}(\tau \sigma^2)^{-rk - \frac k 2} |A_k| ^{\frac 1 2} \exp \left\{- \frac{\bm \beta_k ^ \prime A_k \bm \beta_k}{2 \tau \sigma ^2}\right\}\prod_{i =1}^{k} \beta_{k_i}^{2r}\\
	& \pi(k) \propto c_1^{-k}p^{-c_2k},\\ 
	&  \pi(\tau) = \frac{\left(\frac n 2\right)^{\frac 1 2}}{\Gamma(\frac 1 2)} \tau^{- \frac 3 2} e^{-\frac{n}{2\tau}}.\\
	&   \pi\left(\sigma^2\right) = \frac{\left( \alpha_2 \right)^{\alpha_1}}{\Gamma(\alpha_1)} \left(\sigma^2\right)^{-(\alpha_1 + 1)} e^{-\frac{\alpha_2}{\sigma^2}}.
	\end{split}
	\end{eqnarray}
	\noindent
	where $c_1, c_2>0$ are fixed constants. As indicated in \citep{Castillo:2015}, the rate of decrease for $\pi(k)$ reflects the number of models $\binom{p}{k}$ of given size $k$. Compared to the previous uniform-like prior, these complexity priors are explicitly penalizing larger models to accommodate larger dimensions. In particular, to achieve model selection consistency in this setup, the dimension $p$ can be allowed to grow at a sub-exponential rate of $n$ given in the following condition:
	\begin{condition} \label{assumptionp_new}
		There exist a constant $0 < \kappa < 1$, such that $\log p = O(n^\kappa)$.
	\end{condition} 
	The next result establish the posterior ratio consistency for the complexity prior based approach in (\ref{complexity}).
	\begin{theorem}[Posterior ratio consistency for complexity priors] \label{thm4.1}
		Consider the complexity prior based model described in (\ref{complexity}). Under Assumptions \ref{assumptiontruemodel}, \ref{assumptionmodelsize}, \ref{assumptionhyper} and Condition \ref{assumptionp_new}, the following holds: 
		$$
		\max_{\bm k \ne {\bm t}} \frac{\pi(\bm k|\bm y_n)}{\pi(\bm t|\bm y_n)} \rightarrow 0, \quad \mbox{as } n \rightarrow \infty. 
		$$
	\end{theorem}
		Next, we establish the strong selection consistency result which implies that the posterior mass assigned to the true model $\bm t$ converges to $1$ in probability. 
	\begin{theorem}[Strong selection consistency for complexity priors] \label{thm4}
		Consider the complexity prior based model described in (\ref{complexity}). Under Assumptions \ref{assumptiontruemodel}, \ref{assumptionmodelsize}, \ref{assumptionhyper} and Condition \ref{assumptionp_new}, if we future assume $c_2 > 1$, the following holds:
			$$
		{\pi(\bm t|\bm y_n)} \rightarrow 1, \quad \mbox{as } n \rightarrow \infty. 
		$$
	\end{theorem}
	\noindent
	The proof for Theorem \ref{thm4.1} and \ref{thm4} will also be broken into several steps. The following three lemmas establish the upper bound for the marginal posterior ratio between any ``non-true" model and the true model.
	\begin{lemma} \label{lm6}
		Under Assumptions \ref{assumptiontruemodel}, \ref{assumptionmodelsize}, \ref{assumptionhyper} and Condition \ref{assumptionp_new}, when $\bm k \subset \bm t$, for large enough $n > N_1^{\prime\prime}$ (not depending on $k$), the following holds:
		\begin{align} 
		\frac{\pi(\bm k |\bm y_n)}{\pi(\bm t | \bm y_n)}  \le 2M_1 p^{-2c_2t},
		\end{align}
		for some constants $M_1 > 0$.
	\end{lemma}
	\begin{lemma} \label{lm7}
			Under Assumptions \ref{assumptiontruemodel}, \ref{assumptionmodelsize}, \ref{assumptionhyper} and Condition \ref{assumptionp_new}, When $\bm k \supset \bm t$, for large enough $n > N^{\prime\prime}$ (not depending on $k$), the following holds:
			\begin{align} 
			\frac{\pi(\bm k |\bm y_n)}{\pi(\bm t | \bm y_n)} \le c_1^{-(k-t)}p^{-c_2(k-t)}.
			\end{align}
	\end{lemma} 
	\begin{lemma} \label{lm8}
	Under Assumptions \ref{assumptiontruemodel}, \ref{assumptionmodelsize}, \ref{assumptionhyper} and Condition \ref{assumptionp_new}, when $\bm k \nsubseteq \bm t,$ $\bm k \nsupseteq \bm t$ and $\bm k \neq \bm t$, denote $\bm u = \bm k \cup \bm t$. for large enough $n > N_3^{\prime \prime}$ (not depending on $k$), the following holds:
	\begin{align}
	\frac{\pi(\bm k |\bm y_n)}{\pi(\bm t | \bm y_n)} \le c_3^{-(k-t)}p^{-c_2k},
	\end{align}  
	for some constant $c_3 > 0$. 
\end{lemma} 
\noindent
\begin{proof}[Proof of Theorem \ref{thm4.1} and \ref{thm4}]
Theorem \ref{thm4.1} immediately follows after Lemma \ref{lm6} to \ref{lm8}.      
We now move on to the proof of Theorem \ref{thm4}. 
It follows from Lemma \ref{lm6} to \ref{lm8} that if we restrict to $C_n^c$, then 
\begin{align*}
\frac{1-\pi(\bm t|\bm y_n)}{\pi(\bm t|\bm y_n)} =& \sum_{\bm k \neq \bm t}\frac{\pi(\bm k|\bm y_n)}{\pi(\bm t|\bm y_n)}\\
\le & \sum_{k \le t}\frac{\pi(\bm k|\bm y_n)}{\pi(\bm t|\bm y_n)} + \sum_{k > t, \bm k \supset \bm t}\frac{\pi(\bm k|\bm y_n)}{\pi(\bm t|\bm y_n)} + \sum_{k > t, \bm k \nsupseteq \bm t}\frac{\pi(\bm k|\bm y_n)}{\pi(\bm t|\bm y_n)}\\
\le & \sum_{k = 1}^{t} \binom pk  M_2p^{-2c_2t} + \sum_{k-t = 1}^{q_n-t} \binom {p-t}{k-t} c_1^{-(k-t)}p^{-c_2(k-t)}\\
&+  \sum_{k = 1}^{q_n} \binom {p}{k} c_3^{-(k-t)}p^{-c_2k}.
\end{align*}
By $\binom {p} {k} \le p^{k}$ and $c_2 > 1$, we get
\begin{align*}
\frac{1-\pi(\bm t|\bm y_n)}{\pi(\bm t|\bm y_n)} \rightarrow 0, \mbox{ i.e. } \pi(\bm t|\bm y_n) \rightarrow 1, \quad \mbox{as } n \rightarrow \infty,
\end{align*}
which completes our proof for Theorem \ref{thm4}.
\end{proof}
\begin{remark}
	Note that though under the complexity priors, the restriction on $p$ is relaxed in terms of proving strong selection consistency, we find that in our simulation studies, the model selection performance under uniform-like prior is much better than that under the complexity priors, hence from a practical point of view, one would still prefer the hyper-pMOM with uniform-like prior over the model space. As indicated in \citep{Shin.M:2015}, since the pMOM priors already induce a strong penalty on the model size, it is no longer necessary to penalize larger models through priors on the model space.	
\end{remark}
\section{Computation} \label{sec:computation}
\noindent
The integral formulation in (\ref{model posterior}) is quite complicated, and hence the posterior probabilities can not be obtained in closed form. Hence, we use 
Laplace approximation to compute $m_{\bm k}(\bm y_n)$ and $\pi(\bm k | \bm y_n)$. A similar approach to compute posterior 
probabilities has been used in \citep{Johnson:Rossell:2012} and \citep{Shin.M:2015}. 

Note that for any model $\bm k$, when $A_{k} = I_{k}$, the normalization constant $d_{k}$ in (\ref{modelspecification}) is given by 
$d_{k} = \left((2r-1)!!\right)^{-k}$. Let 
\begin{align} \label{laplacemarginal}
\begin{split}
f(\bm \beta,\tau,\sigma^2) =& \log(m_{\bm k}(\bm y_n)) \\
=& \log\left(\pi(\bm y_n|\sigma^2)\pi(\bm \beta_k|\tau,\sigma^2)\pi(\tau)\pi(\sigma^2)\right)\\
=& -k\log\left((2r-1)!!\right) - \frac {n+k} 2\log(2\pi)  - \left(rk +\frac {n+k} 2 + \alpha_1 + 1\right)\log(\sigma^2) \\
&- \left(rk+\frac {k+3} 2\right)\log\tau -\left(\frac{(\bm y_n - X_k\bm \beta_k)^T(\bm y_n - X_k\bm \beta_k)}{2\sigma^2} \right)\\
&-\left(\frac{\bm \beta_k^T\bm \beta_k }{2\tau\sigma^2} + \frac{\alpha_2}{\sigma^2} + \frac{n}{2\tau}\right)+ \sum_{i=1}^k2r\log(\beta_{k_i})
\end{split}
\end{align}
For any model $\bm k$, the Laplace approximation of $m_{\bm k}(\bm y_n)$ is given by
\begin{align} \label{laplace}
(2\pi)^{\frac k 2 + 1}\exp\left\{f(\hat{\bm \beta_k},\hat{\tau},\hat{\sigma^2})\right\}|V(\hat{\bm \beta_k},\hat{\tau},\hat{\sigma^2})|^{-\frac 1 2},
\end{align}
where $(\hat{\bm \beta_k},\hat{\tau},\hat{\sigma^2}) = \argmax_{(\bm \beta,\tau,\sigma^2)}f(\bm \beta,\tau,\sigma^2)$ obtained via the optimization function nlm in R using a Newton-type algorithm and $V(\hat{\bm \beta_k},\hat{\tau},\hat{\sigma^2})$ is a $(k+2)\times(k+2)$ symmetric matrix with the following blocks:
\begin{align}
\begin{split}
&V_{11} = \frac 1 {\tau\sigma^2}I_k + \frac 1 {\sigma^2}X_k^TX_k + diag\left\{\frac{2r}{\beta_{k_1}^2}, \ldots, \frac{2r}{\beta_{k_k}^2} \right\}, V_{12} = -\frac{\bm \beta_k}{\tau^2\sigma^2},\\
&V_{13} = -\frac{\bm \beta_k}{\tau\sigma^4} - \frac{X_k^TX_k\bm \beta_k - X_k^T\bm y_n}{\sigma^4}, V_{22} = -\frac{rk + \frac k 2 + \frac 3 2}{\tau^2} + \frac{\bm \beta_k^T\bm \beta_k}{\tau^3\sigma^2} + \frac{n}{\tau^3},V_{23} = \frac{\bm \beta_k^T\bm \beta_k}{2\tau^2\sigma^4},\\
&V_{33} = -\frac{rk+\frac k 2 + \frac n 2 + \alpha_1 + 1}{\sigma^4} +\frac{\bm \beta_k^T\bm \beta_k}{\tau\sigma^6} + \frac{(\bm y_n - X_k\bm \beta_k)^T(\bm y_n - X_k\bm \beta_k)}{4\sigma^6} + \frac{2\alpha_2}{\sigma^6}.
\end{split}
\end{align}
The above Laplace approximation can be used to compute the log of the posterior probability ratio between any given model $\bm k$ and true model $\bm t$, and select a model $\bm k$ with the highest probability. Based on a reviewer's comment, we would like to point out that  Laplace approximation could have potential drawbacks. Firstly, as indicated in \cite{Rossell:Telesca:2017}, for non-local priors, Laplace approximations fail to consistently estimate the marginal likelihood for overfitted models. Secondly, the Newton-type algorithm used for optimizing (6.1) could be quite time consuming, especially when the size of the model and the dimension $p$ are large. For example, in Figure 5, the runtime for the hyper-pMOM approach increases as $p$ grows. Despite these potential drawbacks of the Laplace approximation, we would like to point out that in these high-dimensional settings, full posterior sampling using Markov chain Monte Carlo algorithms is highly inefficient and often not feasible from a practical perspective. Hence, the usage of Laplace approximation is still much better than MCMC.

We then adopt the scalable stochastic search algorithm proposed by \cite{Shin.M:2015} called Simplified Shotgun Stochastic Search with Screening (S5). Utilizing the Laplace approximations of the marginal probabilities in (\ref{laplace}), the S5 method aims at rapidly identifying regions of high posterior probability and finding the maximum a posteriori (MAP) model. Detailed algorithm steps can be found in \cite{Shin.M:2015} and the implementation can be found in the R package ``BayesS5". 
\section{Experiments} \label{sec:experiments}

\noindent
\subsection{Simulation I: Illustration of posterior ratio consistency} \label{sec:illustration:posterior:ratio}
In this section, we illustrate the model selection consistency results in Theorems \ref{thm1} and \ref{thm2} 
using a simulation experiment.  The similar simulation setting was also considered in the literature \citep{CKG:2017} by Cao, Khare and Ghosh, in which the authors showed posterior consistency in graphical model setting. We generate our data according to a Gaussian linear model based on the following mechanism. First, we vary $p$ from $500$ to $3000$ and let $n = p/5$. Then, for each fixed $p$, ten covariates are taken as active in the true model with coefficients  $\bm \beta_0 = \left(1.1,1.2,1.3,\ldots,1.9,2\right)^T$ and set $\sigma = 1$. Also, the signs of the true regression coefficients were randomly changed with probability $0.5$. Next, we generate $n$ i.i.d. observations from the $N(\bm 0_p, \Sigma )$ distribution as rows of the covariate matrix $X$. We then examine posterior ratio consistency under three different cases of $\Sigma$ by computing the log posterior ratio of a ``non-true" model $\bm k$ and $\bm t$ as follows. 
\begin{enumerate} 
	\item Case $1$: Isotropic design, where $\Sigma = I_p.$
	\item Case $2$: Compound symmetry design, where $\Sigma_{ij} = 0.5$, if $i \neq j$ and $\Sigma_{ii} = 1,$ for all $1 \le i \le j \le p.$
	\item Case $3$: Autoregressive correlated design; where $\Sigma_{ij} = 0.5^{|i-j|},$ for all $1 \le i \le j \le p$.
\end{enumerate}

\noindent
Throughout this simulation study, we set the hyperparameters $r=2$ and $\alpha_1 = \alpha_2 = 0.01$. The log of the posterior probability ratio for 
various cases of $\Sigma$ is provided in Figure \ref{posterior_ratio_plot}. Note that for each of these cases, we compute the log ratio under four different scenarios of ``non-true'' model $\bm k$. \begin{enumerate}
	\item Scenario $1$: $\bm k$ is a subset of $\bm t$ and $|\bm k| = \frac 1 2 |\bm t|.$
	\item Scenario $2$: $\bm k$ is a superset of $\bm t$ and $|\bm k| = 2 |\bm t|.$
	\item Scenario $3$: $\bm k$ is not necessarily a subset of $\bm t$, but $|\bm k| = \frac 1 2 |\bm t|.$
	\item Scenario $4$: $\bm k$ is not necessarily a superset of $\bm t$, but $|\bm k| = 2 |\bm t|.$
\end{enumerate}
As expected the log of the posterior probability ratio for any ``non-true" model $\bm k$ compared to the true model $\bm t$ are all decreasing to large negative values as $p$ increases, thereby providing a numerical illustration of Theorems \ref{thm1} and \ref{thm2}. 
\begin{figure}[htbp]
	\centering
	\begin{subfigure}
		{\includegraphics[width=35mm]{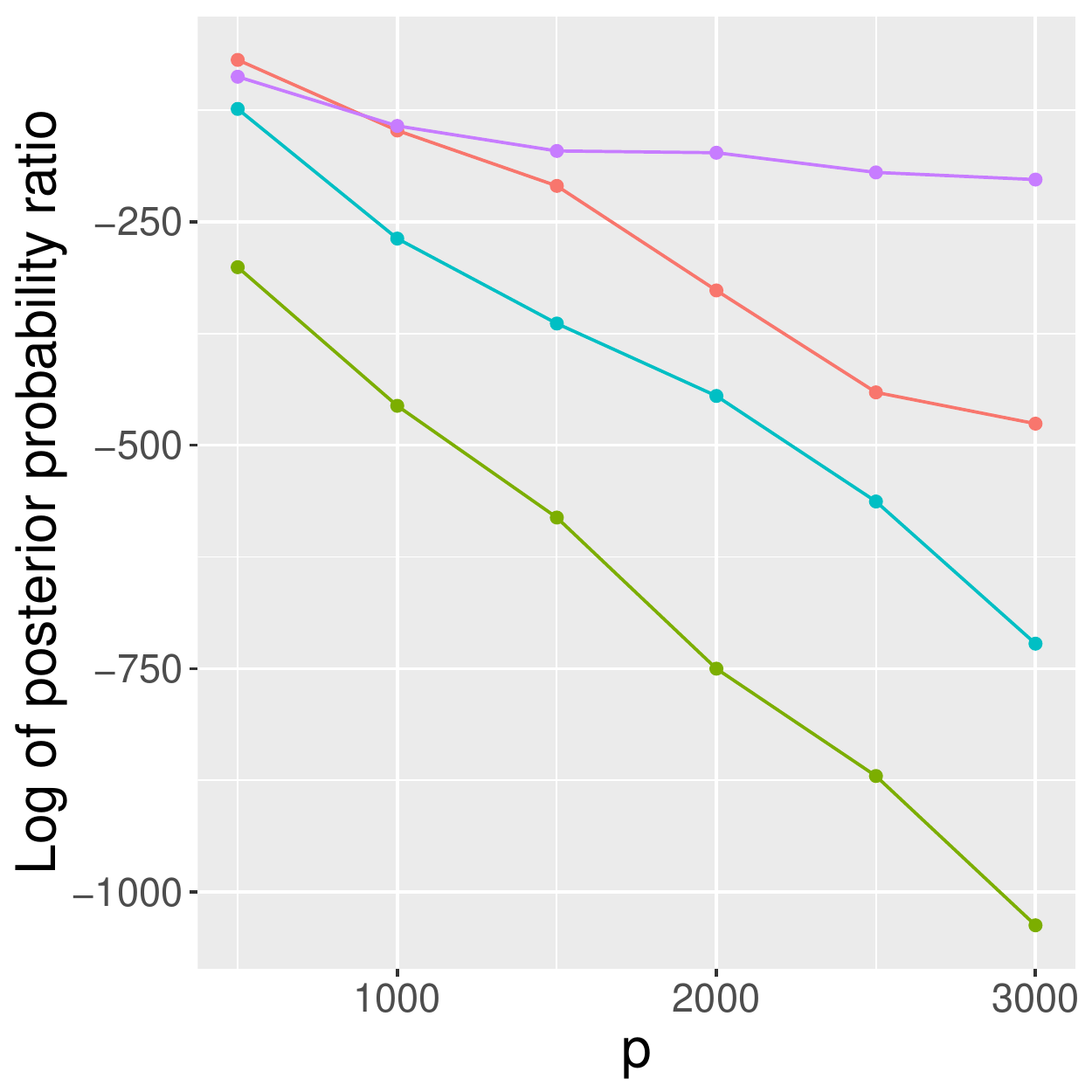}}
	\end{subfigure}
	\begin{subfigure}
		{\includegraphics[width=35mm]{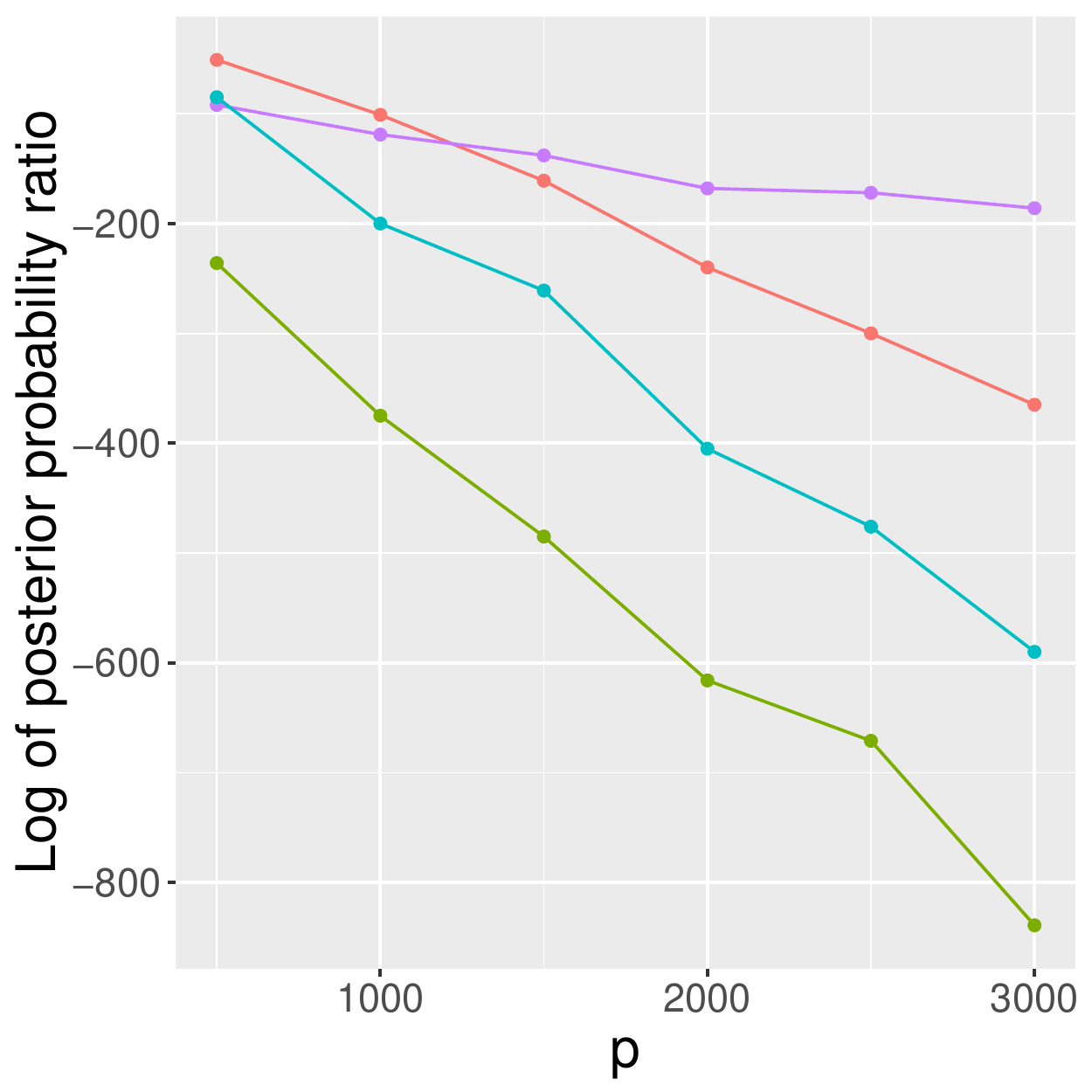}}
	\end{subfigure}
	\begin{subfigure}
		{\includegraphics[width=49mm]{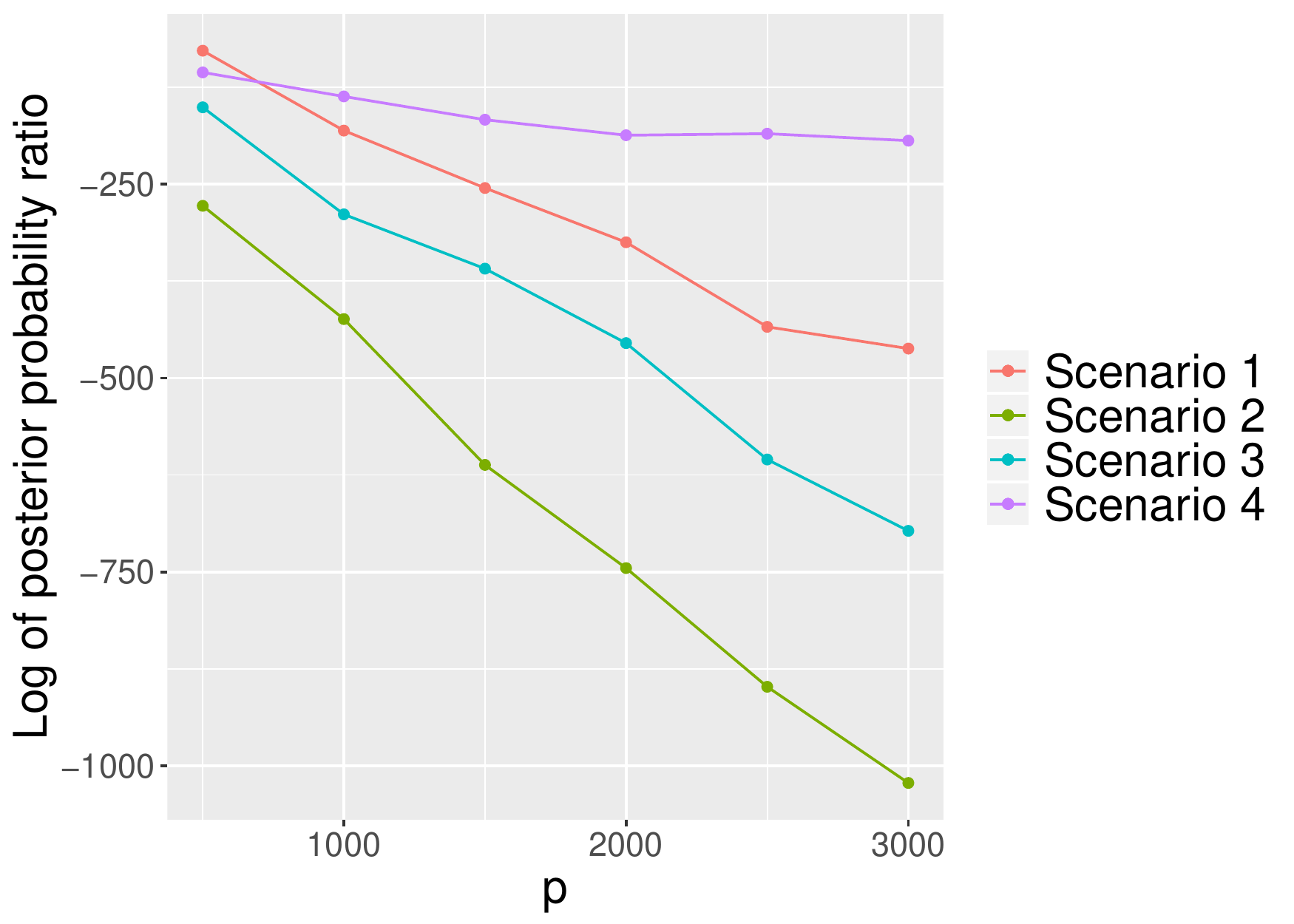}}
	\end{subfigure}
	\caption{Log of posterior probability ratio for $\bm k$ and $\bm t$ for various choices of the ``non-true" model $\bm k$. Left: case $1$; middle: case $2$; right: case $3$. }
	\label{posterior_ratio_plot}
\end{figure}

\subsection{Simulation II: Illustration of model selection} \label{sec:illustration:model:selection}

\noindent
In this section, we perform a simulation experiment to illustrate the potential advantages of using our Bayesian approach. Several different values of $p$ ranging from $500$ to $3000$ are considered, while $n = p/5$. For each fixed $p$, we construct two sets of $\bm \beta_0$. The first set is generated by the same mechanism as in Section \ref{sec:illustration:posterior:ratio}. The other set  also considered is $(0.3,0.35,0.4,0.45,0.5,1.1,1.2,1.3,1.4,1.5)^T$. Next, we generate $n$ i.i.d. observations from the $N(\bm 0_p, \Sigma)$ distribution as rows of covariate matrix $X$ under the following three cases similar to Section \ref{sec:illustration:posterior:ratio}. 
\begin{enumerate} \label{sigma_cases}
	\item Case $1$: Isotropic design, where $\Sigma = I_p.$
	\item Case $2$: Compound symmetry design, where $\Sigma_{ij} = 0.5$, if $i \neq j$ and $\Sigma_{ii} = 1,$ for all $1 \le i \le j \le p.$
	\item Case $3$: Autoregressive correlated design; where $\Sigma_{ij} = 0.5^{|i-j|},$ for all $1 \le i \le j \le p$.
\end{enumerate}

\noindent
Then, we perform model selection using our hierarchical Bayesian approach. This is done by computing the posterior probabilities 
using the Laplace approximation in (\ref{laplace}), and exploring the model space using the simplified stochastic shotgun stochastic 
search algorithm in \citep{Shin.M:2015}.

We would like to remind the readers that in our model, we don't need to specify a fixed value for $\tau$, but rather put a prior on the 
parameter $\tau$ (as opposed to \citep{Johnson:Rossell:2012} and \citep{Shin.M:2015} when $\tau$ is treated as a fixed parameter). 
In Table \ref{model:selection:table:n200} and Table \ref{model:selection:table:beta2}, we also provide model selection performance results with fixed $\tau$ at $\tau = 0.072$ (the default value for the second-order pMOM prior suggested in \cite{Johnson:Rossell:2012}), and numerical values in R package BayesS5 (a choice for fixed $\tau$ from the results in \citep{Shin.M:2015}). Additionally, we also provide model selection performance results for the Lasso \citep{Lasso:1996} and SCAD \citep{SCAD:2001} penalized likelihood methods. 

The model selection performance of these five methods is then compared using several different measures of structure such as 
positive predictive value, true positive rate and false positive rate (average over $20$ independent repetitions). 
Positive Predictive Value (PPV) represents the proportion of true model indexes among all the indexes detected by the given procedure. 
True Positive Rate (TPR) measures the proportion of true indexes detected by the given procedure among all the true indexes
from the true model. False Positive Rate (FPR)  represents the proportion of falsely identified indexes among all the non- true indexes
from the true model. PPV, TPR and FPR are defined as $$\mbox{PPV} = \frac{\text{TP}}{\text{TP + FP}}, \quad \mbox{TPR} = \frac{\text{TP}}{\text{TP + FN}}, \quad \mbox{FPR} = \frac{\text{FP}}{\text{FP + TN}},$$ where TP, TN, FP and FN correspond to true positive, true negative, false positive and false negative, respectively. One would like the PPV and TPR values to be as close to $1$ as possible, while FPR to be as close to $0$ as possible. The results are summarized in Table \ref{model:selection:table:n200} and Table \ref{model:selection:table:beta2}. 

To better visualized the results, in Figure \ref{roc}, we provide the ROC curves when $|\bm \beta_0| =  \left(1.1,1.2,1.3,\ldots,1.9,2\right)^T$ and $\Sigma$ for generating $\bm X$ yields a compound symmetry design. We also include the complexity prior based approach illustrated in Section \ref{sec:complexity}. As we can see, the complexity prior based approach captures fewer true indexes compared to other approaches.

\begin{table}
	\centering
	\scalebox{0.78}{
		\begin{tabular}{cccccccccccccccc}
			\hline
			\multicolumn{1}{c}{\multirow{1}{*}{}}& \multicolumn{3}{c}{Lasso}    & \multicolumn{3}{c}{SCAD}   & \multicolumn{3}{c}{BayesS5}  & \multicolumn{3}{c}{$\tau = 0.072$}  & \multicolumn{3}{c}{Hyper-pMOM}  \\
			p    & PPV       & TPR & FPR        & PPV       & TPR & FPR       & PPV      & TPR & FPR      & PPV     & TPR & FPR    & PPV       & TPR & FPR         \\
			\hline
			200  & 0.29 & 1 & 0.14 & 0.93 & 1 & 0.01 &1  &1  &0  & 0.96 & 1  & 0 & 1         & 1 & 0        \\
			500  & 0.19 & 1 & 0.09 & 0.90 & 1 & 0 & 1 &1  &0  & 0.88 & 1  & 0 & 1         & 1 & 0        \\
			1000 & 0.18 & 1 & 0.04 & 0.87 & 1 & 0 & 0.98 & 1 &0 & 0.69 & 1   & 0 & 1         & 1 & 0        \\
			1500 & 0.18 & 1 & 0.03 & 0.84 & 1 & 0 & 0.98 & 1 &0  & 0.74& 1   & 0 & 1         & 1 & 0        \\
			2000 & 0.17 & 1 & 0.02 & 0.82 & 1 & 0 & 0.98 & 1 & 0 & 0.59 & 1   & 0 & 1         & 1 & 0        \\
			2500 &0.13  & 1 & 0.02 & 0.90 & 1 & 0 & 0.97 & 1 & 0 &0.49 &1 & 0 & 1 & 1 & 0  \\
			\hline
		\end{tabular}
	}
	
	\scalebox{0.78}{
		\begin{tabular}{cccccccccccccccc}
			\hline
			\multicolumn{1}{c}{\multirow{1}{*}{}}& \multicolumn{3}{c}{Lasso}    & \multicolumn{3}{c}{SCAD}   & \multicolumn{3}{c}{BayesS5}  & \multicolumn{3}{c}{$\tau = 0.072$}  & \multicolumn{3}{c}{Hyper-pMOM}  \\
			p    & PPV       & TPR & FPR        & PPV       & TPR & FPR       & PPV      & TPR & FPR      & PPV     & TPR & FPR    & PPV       & TPR & FPR         \\
			\hline
			200  & 0.27 & 1 & 0.15 & 0.96 & 1 & 0 &0.94  & 1 & 0    &0.81  & 1  & 0.01 &1          & 1 & 0           \\
			500  & 0.21 & 1 & 0.09 & 0.94 & 1 & 0  &0.95   &1 &0   &0.59  & 1 & 0.02 &1          & 1 & 0           \\
			1000 & 0.17 & 1 & 0.05 & 0.92 & 1 & 0 &0.95    & 1  & 0 &0.46 & 1 & 0.01  &0.99  & 1 & 0  \\
			1500 & 0.19 & 1 & 0.03 & 0.90 & 1 & 0&0.94  &1 &0   &0.42  & 1 & 0.01 &1  & 1 & 0 \\
			2000 & 0.13 & 1 & 0.04 & 0.84 & 1 & 0 &0.87 &1	&0 &0.41  & 1 & 0.01 & 0.99 & 1 & 0 \\
			2500 &0.12  & 1 & 0.03 & 0.92 & 1 & 0&0.88  & 1 & 0 & 0.36 & 1  & 0.01&0.99 &1 & 0\\ 
			\hline
		\end{tabular}
	}

	\scalebox{0.78}{
		\begin{tabular}{cccccccccccccccc}
			\hline
			\multicolumn{1}{c}{\multirow{1}{*}{}}& \multicolumn{3}{c}{Lasso}    & \multicolumn{3}{c}{SCAD}   & \multicolumn{3}{c}{BayesS5}  & \multicolumn{3}{c}{$\tau = 0.072$}  & \multicolumn{3}{c}{Hyper-pMOM}  \\
			p    & PPV       & TPR & FPR        & PPV       & TPR & FPR       & PPV      & TPR & FPR      & PPV     & TPR & FPR    & PPV       & TPR & FPR         \\
			\hline
			
			200  & 0.25 & 1 & 0.17  & 0.91  & 1& 0 &1 & 1 & 0  &0.96   & 1  & 0&  1        & 1 & 0          \\
			500  & 0.20 & 1 & 0.10  & 0.91 & 1 & 0 & 0.98 &1  &0  &0.83   & 1 & 0& 1         & 1 & 0          \\
			1000 & 0.18 & 1 & 0.05 & 0.85  & 1 & 0 &0.97  &1 &0  &0.73 & 1 & 0 &1 & 1 & 0\\
			1500 & 0.16 & 1 & 0.04 & 0.83 & 1  & 0  &0.96 &1 	&0 & 0.71  & 1 & 0 & 1 & 1 & 0   \\
			2000 & 0.17 & 1 & 0.04 & 0.83 & 1  & 0  &0.96 & 1 &0 &0.57  & 1 & 0 &0.99 & 1 & 0  \\
			2500 & 0.14 & 1 & 0.03 & 0.85 & 1  & 0 &0.96  & 1 & 0 & 0.56 & 1  & 0 & 1 & 1 & 0\\
			\hline
		\end{tabular}
	}
	\caption{Model selection performance comparison table when $|\bm \beta_0| =  \left(1.1,1.2,1.3,\ldots,1.9,2\right)^T$. Top: case $1$; middle: case $2$; bottom: case $3$.}
	\label{model:selection:table:n200}
\end{table}

\begin{figure}[htbp]
	\centering
	\begin{subfigure}
		{\includegraphics[width=35mm]{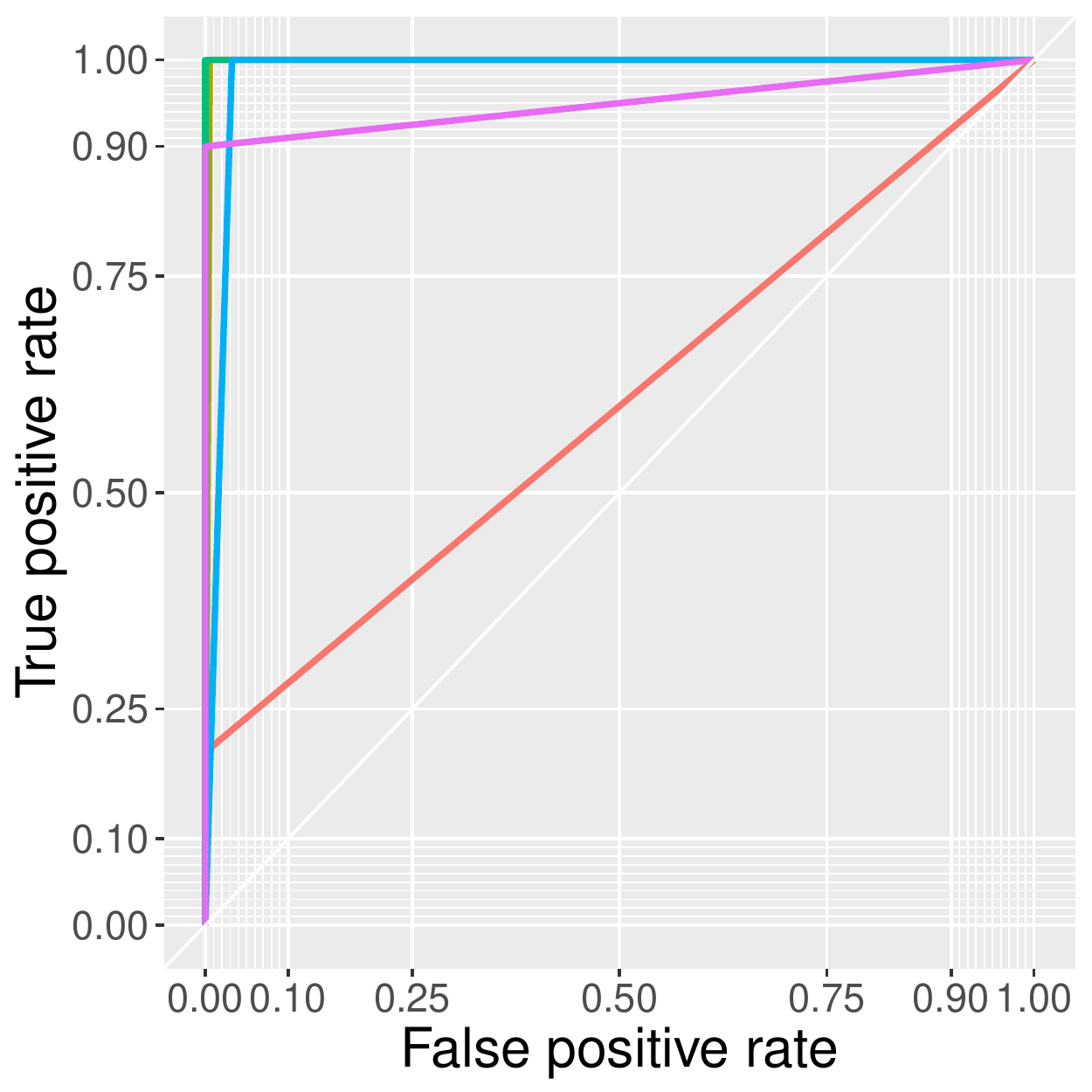}}
	\end{subfigure}
	\begin{subfigure}
		{\includegraphics[width=35mm]{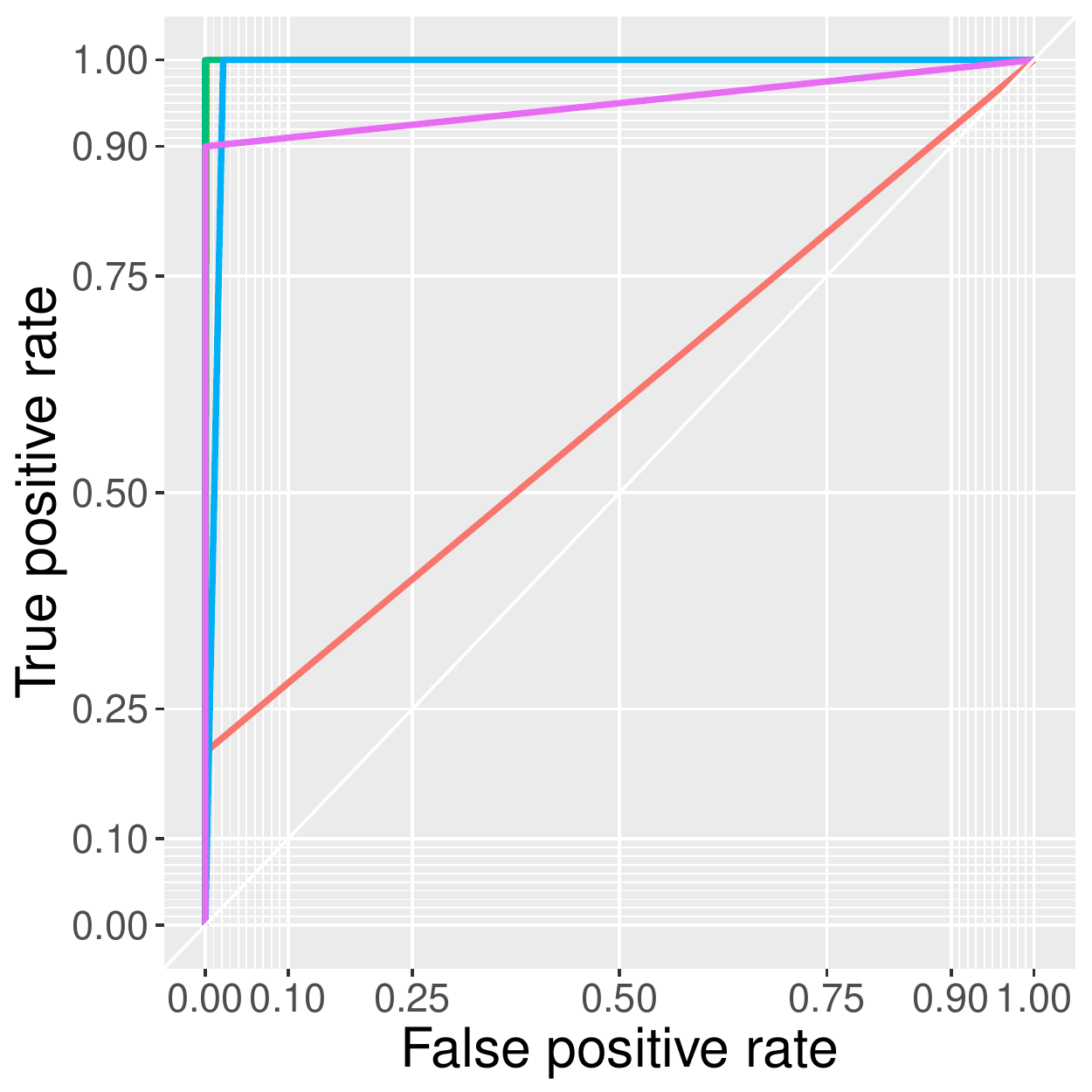}}
	\end{subfigure}
	\begin{subfigure}
		{\includegraphics[width=52.3mm]{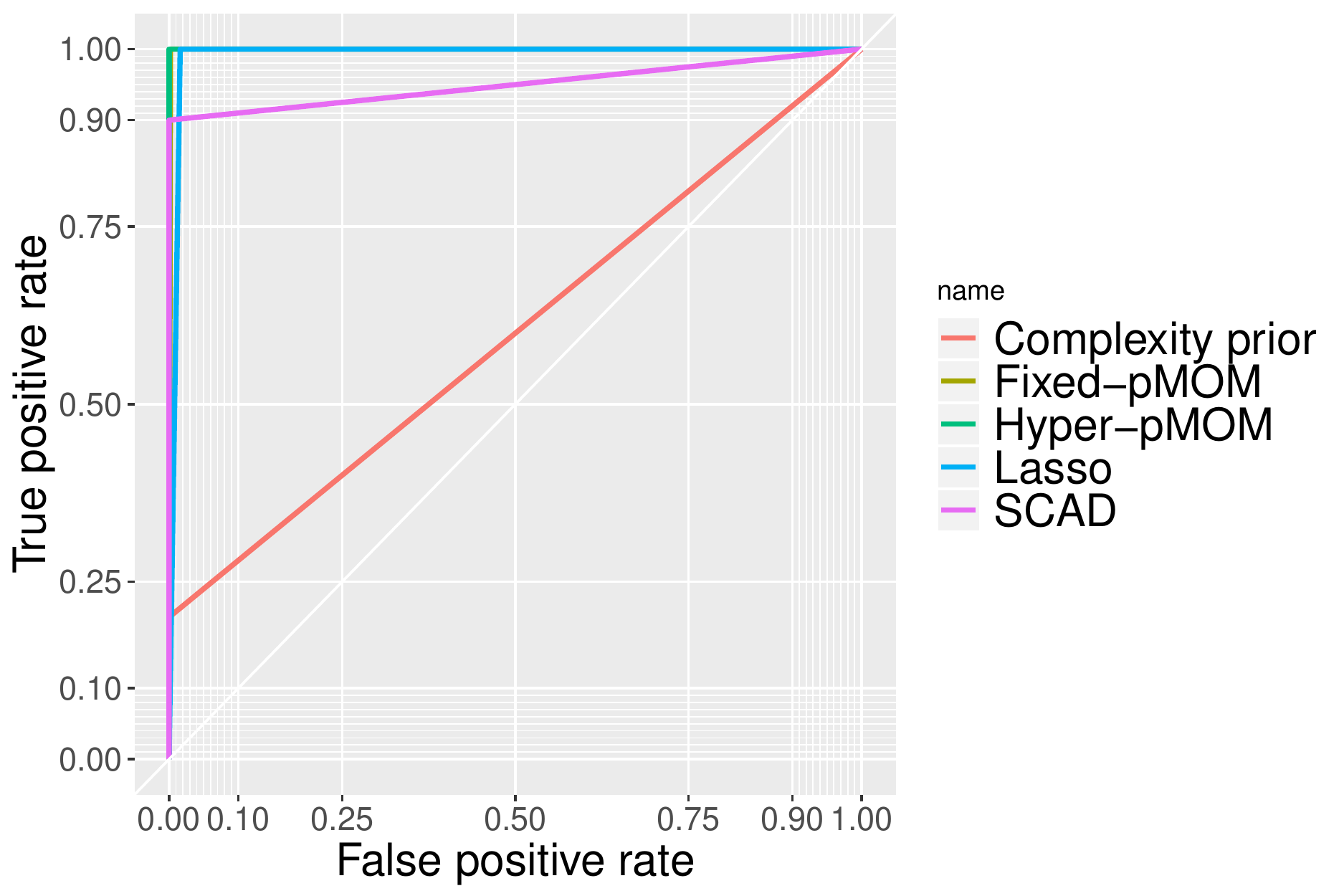}}
	\end{subfigure}
	\caption{ROC curves when $p = 1500$ (left), $p = 2000$ (middle), $p = 2500$ (right).}
	\label{roc}
\end{figure}

\begin{table}
	\centering
	\scalebox{0.78}{
		\begin{tabular}{cccccccccccccccc}
			\hline
			\multicolumn{1}{c}{\multirow{1}{*}{}}& \multicolumn{3}{c}{Lasso}    & \multicolumn{3}{c}{SCAD}   & \multicolumn{3}{c}{BayesS5}  & \multicolumn{3}{c}{$\tau = 0.072$}  & \multicolumn{3}{c}{Hyper-pMOM}  \\
			p    & PPV       & TPR & FPR        & PPV       & TPR & FPR       & PPV      & TPR & FPR      & PPV     & TPR & FPR    & PPV       & TPR & FPR         \\
			\hline
			200  & 0.32 & 1 & 0.22 & 1    & 1    & 0    & 0.97 & 0.95    & 0 &0.95  &0.96    &0  &1   &0.86     & 0 \\
			500  & 0.27 & 1 & 0.06 & 0.48 & 1    & 0.02 & 0.97 & 0.93    & 0 &0.89   &0.93   &0   &1  &0.84    & 0 \\
			1000 & 0.13 & 1 & 0.07 & 0.59 & 1    & 0.01 & 0.95 & 0.95 & 0 &0.88   &0.89   &0   &1  &0.88  & 0 \\
			1500 & 0.23 & 1 & 0.02 & 0.61 & 0.89 & 0    &0.97 & 0.90  & 0 &0.84   &0.90   & 0 &1  &0.88 & 0 \\
			2000 & 0.19 & 1 & 0.03 & 0.63 & 1    & 0    & 0.97 & 0.89 & 0 &0.74   &0.89   &0     &1  &0.84     & 0 \\
			2500 & 0.16 & 1 & 0.03 & 0.59 & 1    & 0.01 & 0.99 & 0.87 & 0 &0.77  &0.88    & 0 &1   &0.83  & 0\\
			\hline
		\end{tabular}
	}
	
	\scalebox{0.78}{
		\begin{tabular}{cccccccccccccccc}
			\hline
			\multicolumn{1}{c}{\multirow{1}{*}{}}& \multicolumn{3}{c}{Lasso}    & \multicolumn{3}{c}{SCAD}   & \multicolumn{3}{c}{BayesS5}  & \multicolumn{3}{c}{$\tau = 0.072$}  & \multicolumn{3}{c}{Hyper-pMOM}  \\
			p    & PPV       & TPR & FPR        & PPV       & TPR & FPR       & PPV      & TPR & FPR      & PPV     & TPR & FPR    & PPV       & TPR & FPR         \\
			\hline
			200  & 0.32 & 1   & 0.12 & 0.88 & 0.7  & 0.01 &0.99     &0.77  & 0 &0.79  &0.84   &0.01  & 1    & 0.81 & 0 \\
			500  & 0.26 & 1   & 0.06 & 1    & 0.83 & 0    &1    &0.72  & 0&0.74  &0.82   &0.01  & 1    & 0.83 & 0 \\
			1000 & 0.19  & 0.89 & 0.02 & 0.57 & 0.81 & 0.01 &1    &0.69   & 0 &0.60  &0.84  &0.01  & 1    & 0.79 & 0 \\
			1500 & 0.19 & 0.91 & 0.03 & 0.57 & 0.80  & 0.05 &0.99     &0.65  & 0 &0.70   &0.80   &0  & 1    & 0.79 & 0 \\
			2000 & 0.17 & 1   & 0.15 & 0.66 & 0.79 & 0.03 &0.95  &0.67  & 0 &0.62  &0.80  &0  & 0.94 & 0.74 & 0 \\
			2500 & 0.18 & 1   & 0.19  & 0.51 & 0.72 & 0.03 & 0.95 & 0.64  & 0 & 0.57 & 0.78 & 0 & 0.95 & 0.70 & 0\\
			\hline
		\end{tabular}
	}

	\scalebox{0.78}{
		\begin{tabular}{cccccccccccccccc}
			\hline
			\multicolumn{1}{c}{\multirow{1}{*}{}}& \multicolumn{3}{c}{Lasso}    & \multicolumn{3}{c}{SCAD}   & \multicolumn{3}{c}{BayesS5}  & \multicolumn{3}{c}{$\tau = 0.072$}  & \multicolumn{3}{c}{Hyper-pMOM}  \\
			p    & PPV       & TPR & FPR        & PPV       & TPR & FPR       & PPV      & TPR & FPR      & PPV     & TPR & FPR    & PPV       & TPR & FPR         \\
			\hline
			200  & 0.37 & 1   & 0.09 & 0.7  & 1    & 0.02 &0.99     &0.92    & 0 & 0.94    & 0.93   & 0  & 1 & 0.90   & 0 \\
			500  & 0.23 & 1   & 0.07 & 0.89 & 0.79  & 0    &0.96    &0.90  & 0 & 0.89 & 0.87 & 0 & 1 & 0.88 & 0 \\
			1000 & 0.13 & 1   & 0.07 & 0.48 & 0.95 & 0.01 &0.96   &0.88  & 0 & 0.77 & 0.86   & 0 & 1 & 0.84 & 0 \\
			1500 & 0.21 & 1   & 0.03 & 0.36 & 0.80  & 0.01 &0.97    &0.87  & 0 & 0.75 & 0.86   & 0 & 1 & 0.89  & 0 \\
			2000 & 0.16 & 0.9 & 0.03 & 0.35 & 0.71  & 0.01 &0.95    &0.88  & 0  &0.84   &0.82   & 0 & 1 & 0.86 & 0 \\
			2500 & 0.13 & 1   & 0.03 & 0.45 & 0.68 & 0    &0.95  &0.81 & 0 &0.80   &0.82   &0   & 1 & 0.78 & 0\\
			\hline
		\end{tabular}
	}
	\caption{Model selection performance comparison table when $\bm \beta_0 = (0.3,0.35,0.4,0.45,0.5,1.1,1.2,1.3,1.4,1.5)^T$. Top: case $1$; middle: case $2$; bottom: case $3$.}
	\label{model:selection:table:beta2}
\end{table}

Based on Table \ref{model:selection:table:n200} and \ref{model:selection:table:beta2}, it is clear that our Bayesian approach outperforms both the penalized likelihood approaches and the fixed $\tau$ setting based on almost all measures and under all cases. The PPV values for our hyper-pMOM approach are all higher than the other four methods, which means our method can identify the true model more precisely. In addition, The FPR values for the Bayesian approach are all significantly smaller than the FPR values for the penalized approaches. It is also worth noting that especially in lower dimensions, the numerical procedure for choosing $\tau$ implemented in BayesS5 needs additional run time as shown in Figure \ref{fig:runtime}, while in our simulation studies, not only this step is omitted, we are still able to better simulation results.  Overall, this experiment illustrates the fact that the Bayesian approach can lead to a significant improvement in model selection performance as compared to penalized likelihood methods. Also, the hierarchical Bayesian approach introduced in this paper can lead to a significant improvement in performance as compared to the fixed $\tau$ Bayesian approach when sample size is much smaller than the number of predictors. 
\begin{figure} \label{fig:runtime}
	\centering
	\includegraphics[width=75mm]{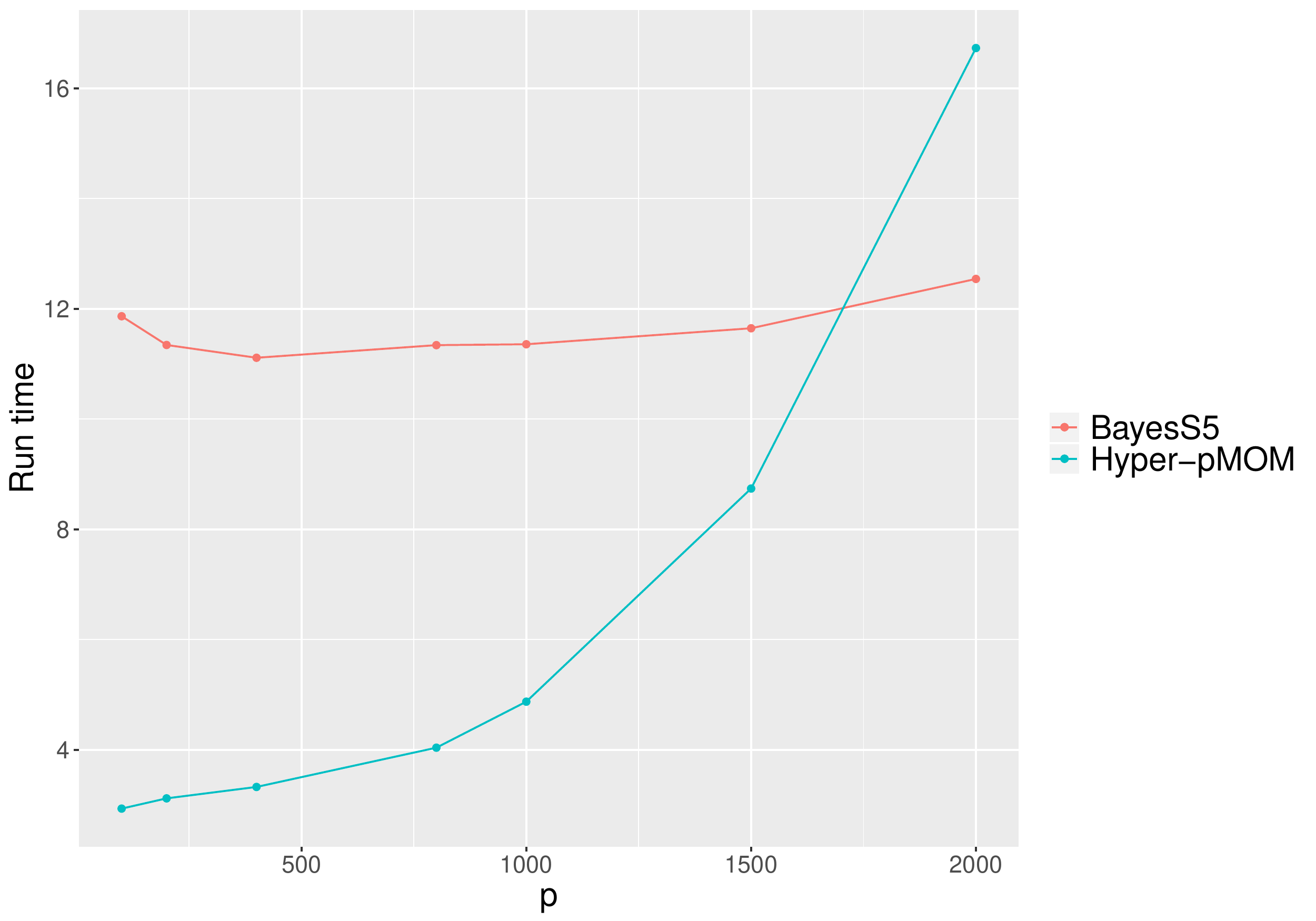}
		\caption{Run time comparison in seconds.}
\end{figure}                                                                                                            
\section{Real Data Analysis} \label{sec:real}
In this section, we carry out the real data analysis to examine the performance of proposed method based on the Boston housing dataset. The dataset contains the median value of owner-occupied homes in the Boston region as the responsive variable, together with several other possible predictor variables including the geographical characteristics. The total number of observations is $n = 506$ and 10 continuous variables: crim, indus, nox, rm, age, dis, tax, ptratio, b, and lstat are considered as the predictor variables.  Several approaches for variable selection have been demonstrated via this housing dataset. See for example \citep{yuan:lin:2005, Shin.M:2015}.

We added 1000 noise variables generated independently from a standard normal distribution to perform the model selection in a $p > n$ regression setting. The design matrix is standardized and the dataset is divided into a training set of size $406$ and a test set of size $100$. We first obtain the model estimate based on the training set and then compare the proposed hyper-pMOM approach with the following four methods on the test set: pMOM with fixed $\tau = 0.072$, peMOM with simplified shotgun stochastic search, and two frequentist approaches, Lasso and SCAD. 

The results are summarized in Table \ref{table:boston} averaged over 100 repetitions based on the following five measures also adopted in \citep{Shin.M:2015}. MSPE represents the out-of-sample square prediction error calculated by
\begin{equation*}
\mbox{MSPE} = \frac 1 {100} \sum_{i \in test}\left(y_i - X_i^T\hat {\bm \beta}_{\hat k}^{train}\right)^2,
\end{equation*}
where $\hat {\bm \beta}_{\hat k}^{train}$ is the least squared estimator based on the model estimate obtained from the test set. MS-O and MS-N refer to the average
original variables and falsely selected noise variables over 100 repetitions, respectively. FS-O is the number of original variables that are selected at least 95 out of 100 repetitions. TS-O refers to the number of original variables that are selected at least once from 100 repetitions.

As we see in Table \ref{table:boston}, our hyper-pMOM approach consistently identifies the same model and had the lowest prediction error among all the five methods. In particular, the average number of the original variables that are selected at least 95 times is 3. Across all the 100 repetitions, our hyper-pMOM method successfully avoids selecting any noise variable, while all the other four methods falsely identify at least one noise variable. Overall, the real data application illustrates our hyper-pMOM approach yields the most stable and accurate model selection among all the five methods.
\begin{table}\label{table:boston}
	\begin{tabular}{cccccc}
		\hline
	\text{  }	&MSPE  & MS-O        & MS-N       & FS-O         & TS-O    \\
		\hline
		Hyper-pMOM & 17.53 & 3   & 0     & 3    & 3 \\
		pMOM  & 21.57 & 5 & 4   & 5    & 4 \\
		peMOM & 18.08 & 5  & 1& 5    & 5 \\
		Lasso & 23.40 & 5.58& 17.81& 4    & 6 \\
		SCAD & 22.83 & 4.98& 12.89& 5    & 5\\
		\hline
	\end{tabular}
\caption{Model selection comparison based on the Boston housing data.}
\end{table}
\section{Discussion} \label{sec:discussion}
This article describes and examines theoretical properties of hyper-pMOM priors proposed in \citep{PhDthesis:Wu} for variable selection in high-dimensional linear model settings. Under standard regularity assumptions, which include the prior over all models is restricted to ones with model size less than an appropriate function of the sample size $n$, we establish posterior ratio consistency (Theorem \ref{thm1}), i.e., the ratio of the 
maximum marginal posterior probability assigned to a ``non-true" model to the posterior probability assigned to the ``true" 
model converges to zero in probability. 
Next, under the additional assumption that $p$ increases 
at a polynomial rate with $n$, we show strong model selection consistency (Theorem \ref{thm2}). Strong model selection 
consistency implies that the posterior probability of the true model converges in probability to $1$ as $n \rightarrow \infty$. 

 Based on the reviewers' comments, we realize the polynomial rate restriction on $p$ could be rather limited. By carefully examining our theoretical analysis, in Section 5, we add another result where we replace the uniform-like prior with the complexity prior on the model space to penalize larger models, and establish strong model selection consistency (Theorem \ref{thm4}) when $p$ is allowed to grow at a sub-exponential rate of $n$. However, through simulation studies, we find out that the model selection performance under the uniform-like prior is much better than that under the complexity prior, hence from a practical point of view, one would still prefer the hyper-pMOM with uniform-like prior on the model space.

In Section \ref{sec:computation}, we provide details about the application of Laplace approximation to approximate the posterior density and illustrate the potential benefits for our hyper-pMOM based model selection procedure compared with other methods via simulation studies and real data analysis in Section \ref{sec:experiments} and Section \ref{sec:real}, respectively.

\bibliographystyle{ba}
\bibliography{references}

\renewcommand{\theequation}{A.\arabic{equation}}
\setcounter{equation}{0}

\renewcommand\thesection{\Alph{section}}
\section{Proof of Lemma \ref{chisquaredtail}}
First note that 
\begin{align*}
P(\chi_p^2 - p > a) &\le \inf_{0 < t < \frac 1 2}\left[\exp\left\{-t(p+a)\right\}E\left(t\chi_p^2\right)\right]\\
&=  \inf_{0 < t < \frac 1 2}\left[\exp\left\{-t(p+a)\right\}(1-2t)^{-\frac p 2}\right].
\end{align*}
Let $g(t)  = -t(p+a) - \frac p 2\log(1-2t),$ then $g^\prime(t) = -(p+a) + \frac{p}{1-2t}$ and $g^{\prime\prime}(t) = \frac{2p}{(1-2t)^2} > 0.$ Hence, $g(t)$ is minimized at $t = t_0 = \frac a {2(p+a)} < \frac 1 2.$ Hence,
\begin{align*}
P(\chi_p^2 - p > a) &\le \inf_{0 < t < \frac 1 2}\left[\exp\left\{-t(p+a)\right\}(1-2t)^{-\frac p 2}\right]\\
&= \exp\left(-\frac a 2\right)\left(\frac p {p+a}\right)^{-\frac p 2}\\
&= \exp\left[-\frac p 2\left\{\frac a p - \log(1+\frac a p)\right\}\right].
\end{align*}
From the inequality, $\log\left(1+\frac a p\right) \le \frac a p - \frac{a^2}{2p(p+a)},$ we get,
$$P\left(\chi_p^2 - p > a\right) \le \exp\left(-\frac{a^2}{4(p+a)}\right).$$
Similarly, we get $P\left(\chi_p^2 - p < -a\right) \le \exp\left(-\frac{a^2}{4(p+a)}\right).$ Therefore,
$$P\left(\lvert\chi_p^2 - p\rvert > a\right) \le 2\exp\left(-\frac{a^2}{4(p+a)}\right).$$
Next, Note that $X \sim \chi_p^2(\lambda)$ is equivalent to $X | K = k \sim \chi_{p+2k}^2$ and $K \sim Poisson(\frac1 2 \lambda).$ Thus, $E(X) = EE(X|K) = p+\lambda$ and the moment-generating function of $X$ is given by
\begin{align*}
E\left[\exp(tX)\right] =& EE\left[\exp(tX)|K\right]\\
=& E\left[(1-2t)^{-\frac{p+2K}{2}}\right]\\
=&(1-2t)^{-\frac p 2}\exp\left[\frac {\lambda t}{2(1-2t)}\right], 0<t< \frac 1 2.
\end{align*}
Hence, for $a>0$,
\begin{align} \label{noncentralproof}
\begin{split}
&P(\chi_p^2(\lambda) - (p + \lambda) > a) \\
\le & \inf_{0 < t < \frac 1 2}\left[\exp\left\{-t(p+\lambda+a)+\frac{\lambda t}{1-2t}\right\}(1-2t)^{-\frac p 2}\right].
\end{split}
\end{align}
Let $g(t) = -t(p+\lambda+a) + \frac{\lambda t}{1-2t} - \frac p 2\log(1-2t).$ Since $g^{\prime \prime}(t) > 0$, $g(t)$ is minimized at $t = t_0 = \frac a {2(p+\lambda+a)} < \frac 1 2$ where $g^\prime(t_0) = 0.$ \\
Now, for $a>0$, by (\ref{noncentralproof}),
\begin{align*}
&P\left(\chi_p^2(\lambda) - (p + \lambda) > a\right) \\
\le & \exp\left[-t_0(p+\lambda +a) + \frac{\lambda t_0}{1-2t_0}\right](1-2t_0)^{-\frac p 2}\\
\le & \exp\left(-\frac p 2\left\{\frac a {p+\lambda} - \log\left(1+\frac a {p+\lambda}\right)\right\}\right).
\end{align*}

\section{Proof of Lemma \ref{lm1}}
It follows by the form of the non-local priors in (\ref{modelspecification}), and by Assumptions \ref{assumptiontruemodel} and \ref{assumptionhyper} that, 
\begin{align}
\left(\frac{a_1}{a_2}\right)^{\frac k 2}\left(\frac{a_1^r}{(2r-1)!!}\right)^k < d_k < \left(\frac{a_2}{a_1}\right)^{\frac k 2}\left(\frac{a_2^r}{(2r-1)!!}\right)^k
\end{align}
and
\begin{align}\label{lm1proof1}
(n\epsilon_n^{-1} + \frac{a_2}{\tau})^{\frac{k}{2}} > |C_k|^{\frac 1 2} > (n\epsilon_n)^{\frac k 2}.
\end{align}

Hence, by (\ref{marginal density}) and (\ref{lm1proof1}),
\begin{align} \label{lm1mk}
\begin{split}
&m_{\bm k}(\bm y_n)\\
<& \frac{\frac{\left(\frac n 2\right)^{\frac 1 2}}{\Gamma(\frac 1 2)}}{(\sqrt{2\pi})^n}\frac{\left( \alpha_2 \right)^{\alpha_1}}{\Gamma(\alpha_1)}\left(\frac{{a_2^{1+r}}}{\sqrt{ca_1}(2r - 1)!!}\right)^k (n\epsilon_n)^{- \frac k 2}\\
&\times \int_{0}^{\infty}\int_{0}^{\infty} (\sigma^2)^{-\left(\frac n 2 + rk + \alpha_1 + 1\right)} \exp\left\{-\frac{R_k + 2\alpha_2}{2\sigma^2}\right\} \tau^{-rk - \frac k 2 - \frac 3 2}e^{-\frac{n}{2\tau}}E_k(\prod_{i=1}^{k}\beta_{k_i}^{2r})d\sigma^2 d\tau.
\end{split}
\end{align}
When $\bm k \neq \bm t,$ by Lemma 6 in the supplementary material for \citep{Johnson:Rossell:2012}, we get 
\begin{align*}
Q_k = E_k(\prod_{i=1}^{k}\beta_{k_i}^{2r}) &< \left(\frac{\epsilon_n^{-1}+ \frac{a_2}{n\tau}}{\epsilon_n} \right)^{\frac k 2}\left(\frac{4V}{k} + \left[(2r-1)!!\right]^{\frac 1 r}\frac{4\sigma^2}{\epsilon_n n}\right)^{rk}\\
&\le  \left(\frac{1 + \frac{a_2\epsilon_n}{n\tau}}{\epsilon_n^2} \right)^{\frac k 2}2^{rk-1}\left(\left(\frac{4V}{k}\right)^{rk} + \left(\frac{4\sigma^2(2r-1)!!}{\epsilon_n n}\right)^{rk}\right).
\end{align*}
Note that $V \le \frac{1}{n\epsilon_n^5}\bm y_n^TP_u\bm y_n$ under Assumption \ref{assumptiontruemodel}, where $P_u = X_u(X_u^TX_u)^{-1}X_u^T$.\\
Therefore, from (\ref{lm1mk}), for large enough constants $M_1,M_2$, we get
\begin{align*}
&m_{\bm k}(\bm y_n) \\
<& \frac{\sqrt n}{\left(\sqrt{2\pi}\right)^n} M_1 M_2^k \left(\frac{V}{k}\right)^{rk}n^{- \frac k 2}\epsilon_n^{-\frac 3 2 k}\\
&\times \int_{0}^{\infty}\int_{0}^{\infty} (\sigma^2)^{-\left(\frac n 2 + rk + \alpha_1 + 1\right)} \exp\left\{-\frac{R_k + 2\alpha_2}{2\sigma^2}\right\} \tau^{-rk - \frac k 2 - \frac 3 2}e^{-\frac{n}{2\tau}}\left(1+\frac{a_2\epsilon_n}{n\tau}\right)^{\frac k 2}d\sigma^2 d\tau\\
&+  \frac{\sqrt n}{\left(\sqrt{2\pi}\right)^n} M_1 M_2^k n^{- \frac k 2 - rk} \epsilon_n^{-\frac 3 2 k - rk}\\
&\times \int_{0}^{\infty}\int_{0}^{\infty} (\sigma^2)^{-\left(\frac n 2  + \alpha_1 + 1\right)} \exp\left\{-\frac{R_k+ 2\alpha_2}{2\sigma^2}\right\} \tau^{- rk - \frac k 2 - \frac 3 2}e^{-\frac{n}{2\tau}}\left(1+\frac{a_2 \epsilon_n}{n\tau}\right)^{\frac k 2}d\sigma^2 d\tau.
\end{align*}
Since $1 + x \le e^x,$ integrating out both $\tau$ and $\sigma^2$ gives us
\begin{align} \label{mk2}
\begin{split}
m_{\bm k}(\bm y_n) \le& \frac{\sqrt n}{\left(\sqrt{2\pi}\right)^n} M_1 M_2^k\left(\frac{V}{k}\right)^{rk}n^{- \frac k 2}\epsilon_n^{-\frac 3 2 k}\\
&\times \frac{\Gamma\left(rk + \frac k 2 +\frac 1 2 \right)\Gamma(\frac n 2 + rk + \alpha_1)}{\left(\frac n 2  - \frac {a_2k}{2n\epsilon_n^{-1}}\right)^{rk + \frac k 2 + \frac 1 2}}2^{\frac n 2 + rk + \alpha_1}\frac {1}{\{R_k^*+ 2\alpha_2\}^{\frac n 2 + rk + \alpha_1}}\\
&+  \frac{\sqrt n}{(\left(\sqrt{2\pi}\right)^n} M_1 M_2^k n^{- \frac k 2 - rk} \epsilon_n^{-\frac 3 2 k - rk}\\
&\times \frac{\Gamma\left(rk + \frac k 2 +\frac 1 2 \right)\Gamma(\frac n 2  + \alpha_1)}{\left(\frac n 2  - \frac {a_2k}{2n\epsilon_n^{-1}}\right)^{rk + \frac k 2 + \frac 1 2}}2^{\frac n 2  + \alpha_1}\frac {1}{\{R_k^*+ 2\alpha_2\}^{\frac n 2  + \alpha_1}}.
\end{split}
\end{align}

\noindent
Similarly for the true model $\bm t$, using (\ref{marginal density}) and (\ref{lm1proof1}), $1 + x \le e^x$, and first integrating out 
$\sigma^2$, we get 
\begin{align} \label{mt}
\begin{split}
m_{\bm t}(\bm y_n) >&\frac{\sqrt n}{\left(\sqrt{2\pi}\right)^n} M_1 M_2^t (n\epsilon_n)^{-\frac t 2} \\
&\times \int_{0}^{\infty}\int_{0}^{\infty}(\sigma^2)^{-\left(\frac n 2 + rt + \alpha_1 + 1\right)} \exp\left\{-\frac{R_t + 2\alpha_2}{2\sigma^2}\right\} \tau^{-rt - \frac t 2 - \frac 3 2}e^{-\left(\frac{n}{2} + \frac{a_2t}{2n\epsilon_n^{-1}} \right) \frac1 \tau }d\sigma^2 d\tau\\
\ge&\frac{\sqrt n}{\left(\sqrt{2\pi}\right)^n} M_1 M_2^t (n\epsilon_n)^{-\frac t 2}\Gamma(\frac n 2 + rt + \alpha_1)2^{\frac n 2 + rt + \alpha_1} \\
&\times \int_{0}^{\infty} \frac 1{\{R_t + 2\alpha_2\}^{\frac n 2 + rt + \alpha_1}} \tau^{-rt - \frac t 2 - \frac 3 2}e^{-\left(\frac{n}{2} + \frac{a_2t}{2n\epsilon_n^{-1}} \right) \frac 1 \tau } d\tau.\\
\end{split}
\end{align}
Note that
\begin{align*}
\{R_t + 2\alpha_2\}^{\frac n 2 + rt + \alpha_1}\le& \{R_t^* + 2\alpha_2\}^{\frac n 2 + rt + \alpha_1}\left\{1 + \frac{R_t - R_t^* }{R_t^*}\right\}^{\frac n 2 + rt + \alpha_1},
\end{align*}
and, by Assumptions \ref{assumptiontruemodel} and \ref{assumptionhyper},
\begin{align*}
&\frac{R_t - R_t^* }{R_t^*}\\
\le& \frac 1{R_t^*}\bm y_n^TX_t(X_t^TX_t)^{-\frac 1 2}\left(I - \left(I + \frac{(X_t^TX_t)^{-\frac 1 2}A_t(X_t^TX_t)^{-\frac 1 2}}{\tau}\right)^{-1}\right)(X_t^TX_t)^{-\frac 1 2}X_t^T\bm y_n\\
\le&  \frac{\frac{a_2}{n\epsilon_n}}{\frac{a_2}{n\epsilon_n} + \tau}\frac{\bm y_n^TX_t(X_t^TX_t)^{-1}X_t^T\bm y_n}{R_t^*}.
\end{align*}

\noindent
From (\ref{mt}), we have
\begin{align} \label{mt1}
\begin{split}
m_{\bm t}(\bm y_n) >& \frac{\sqrt n}{\left(\sqrt{2\pi}\right)^n} M_1 M_2^t  (n\epsilon_n)^{-\frac t 2}\Gamma(\frac n 2 + rt + \alpha_1) \frac {2^{\frac n 2 + rt + \alpha_1}}{\{R_t^* + 2\alpha_2\}^{\frac n 2 + rt + \alpha_1}}\\
&\times \int_{0}^{\infty} \exp \left[ - \left(\frac n 2 + rt + \alpha_1\right)\frac{\frac{a_2}{n\epsilon_n}}{\frac{a_2}{n\epsilon_n} + \tau}\frac{\bm y_n^TP_t\bm y_n}{R_t^*}\right]\tau^{-rt - \frac t 2 - \frac 3 2}e^{-\left(\frac{n}{2} + \frac{a_2t}{2n\epsilon_n^{-1}} \right) \frac 1 \tau } d\tau.\\	
\end{split}
\end{align}
Note that $P_t \le P_u$. By (\ref{chisquaredRt}) and (\ref{chisquaredPk}), there exists $N_1$, such that for all $n>N_1$ (not depending on the model),
\begin{align} \label{lm1proofmt1}
\left(\frac n 2 + rt + \alpha_1\right)\frac{\frac{a_2}{n\epsilon_n}}{\frac{a_2}{n\epsilon_n} + \tau}\frac{\bm y_n^TP_t\bm y_n}{R_t^*} \le& \frac{\frac {a_2\sigma_0^2}{\epsilon_n}\log n}{\tau} \frac{\frac n 2 +rt + \alpha_1}{\sigma_0^2(n - t - \sqrt{n-t}\log n)} < \frac {a_2}{\epsilon_n}\frac {\log n} {\tau}.
\end{align}
Using (\ref{mt1}) and (\ref{lm1proofmt1}), integrating out $\tau,$ we get
\begin{align} \label{mt2}
\begin{split}
m_{\bm t}(\bm y_n) >& \frac{\sqrt n}{\left(\sqrt{2\pi}\right)^n} M_1 M_2^t  (n\epsilon_n)^{-\frac t 2}\Gamma(\frac n 2 + rt + \alpha_1) \frac 1{\{R_t^* + 2\alpha_2\}^{\frac n 2 + rt + \alpha_1}}2^{\frac n 2 + rt + \alpha_1}\\
&\times \int_{0}^{\infty}\tau^{-rt - \frac t 2 - \frac 3 2}e^{-\left(\frac{n}{2} + \frac {a_2\log n}{\epsilon_n} + \frac{a_2t}{2n\epsilon_n^{-1}} \right) \frac 1 \tau } d\tau\\
\ge& \frac{\sqrt n}{\left(\sqrt{2\pi}\right)^n} M_1 M_2^t (n\epsilon_n)^{-\frac t 2}\\
&\times  \frac{\Gamma\left(rt + \frac t 2 +\frac 1 2 \right)\Gamma(\frac n 2 + rt + \alpha_1)}{\left(\frac n 2 + \frac {a_2\log n}{\epsilon_n} + \frac {a_2t}{2n\epsilon_n^{-1}}\right)^{rt + \frac t 2 + \frac 1 2}}2^{\frac n 2 + rt + \alpha_1}\frac {1}{\{R_t^* + 2\alpha_2\}^{\frac n 2 + rt + \alpha_1}}.
\end{split}
\end{align}
It follows from (\ref{mk2}) and Assumption \ref{assumptiontruemodel} that, the upper bound for $m_{\bm t}(\bm y_n)$ will be given by,
\begin{align}  \label{lm1mkt} 
\begin{split}
\frac{m_{\bm k}(\bm y_n)}{m_{\bm t}(\bm y_n)} <& B_1{A_1}^k\left(\frac{V}{k\epsilon_n^2}\right)^{rk}n^{- \frac 1 2 (k-t)}\frac{\Gamma\left(rk + \frac k 2 +\frac 1 2 \right)\Gamma(\frac n 2 + rk + \alpha_1)}{\Gamma(\frac n 2 + rt + \alpha_1)}\\
&\times\frac{\left(\frac n 2 + \frac {a_2\log n}{c} + \frac {a_2\epsilon_nt}{2n}\right)^{rt + \frac t 2 + \frac 1 2}}{\left(\frac n 2  - \frac {a_2\epsilon_nk}{2n}\right)^{rk + \frac k 2 + \frac 1 2}} \frac{\{R_t^* + 2\alpha_2\}^{\frac n 2 + rt + \alpha_1}}{\{R_k^* + 2\alpha_2\}^{\frac n 2 + rk + \alpha_1}}\\
&+ B_1{A_1}^kn^{- \frac 1 2(k - t) - \frac 3 4 rk} \frac{\Gamma\left(rk + \frac k 2 +\frac 1 2 \right)\Gamma(\frac n 2 + \alpha_1)}{\Gamma(\frac n 2 + rt + \alpha_1)}\\
&\times \frac{\left(\frac n 2 + \frac {a_2\log n}{c} + \frac {a_2\epsilon_n t}{2n}\right)^{rt + \frac t 2 + \frac 1 2}}{\left(\frac n 2  - \frac {a_2\epsilon_n k}{2n}\right)^{rk + \frac k 2 + \frac 1 2}}\frac{\{R_t^* + 2\alpha_2\}^{\frac n 2 + rt + \alpha_1}}{\{R_k^* + 2\alpha_2\}^{\frac n 2 + \alpha_1}},
\end{split}
\end{align}
where $A_1$ and $B_1$ are large enough constants.\\
\noindent
Using Stirling approximation, there exists $N_2$ (not depending on $k$), such that for all $n>N_2$ and a large enough constant $M^\prime$, we have
\begin{equation*}
\frac{\Gamma\left(rk + \frac k 2 +\frac 1 2 \right)\Gamma(\frac n 2 + rk + \alpha_1)}{\Gamma(\frac n 2 + rt + \alpha_1)} \le M^\prime k^{(r+1)k}n^{r(k - t)},
\end{equation*}
and 
\begin{equation*}
\frac{\Gamma\left(rk + \frac k 2 +\frac 1 2 \right)\Gamma(\frac n 2  + \alpha_1)}{\Gamma(\frac n 2 + rt + \alpha_1)} \le M^\prime k^{(r+1)k}n^{-rt}.
\end{equation*}
By (\ref{lm1mkt}), there exists $N = \max\{N_1,N_2\},$ such that for $n > N$ (not depending on $k$), we get
\begin{align} \label{mkt2}
\begin{split}
\frac{m_{\bm k}(\bm y_n)}{m_{\bm t}(\bm y_n)} <& BA^k\left(\frac V{\epsilon_n^2}\right)^{rk}k^{k}n^{-(k-t)}\frac{\{R_t^* + 2\alpha_2\}^{\frac n 2 + rt + \alpha_1}}{\{R_k^* + 2\alpha_2\}^{\frac n 2 + rk + \alpha_1}}\\
&+ BA^kk^{(r+1)k}n^{-(r+1)(k-t)-\frac 3 4 rk - rt} \frac{\{R_t^* + 2\alpha_2\}^{\frac n 2 + rt + \alpha_1}}{\{R_k^* + 2\alpha_2\}^{\frac n 2 + \alpha_1}},
\end{split}
\end{align}
where $A = A_1 \times M^\prime$ and $B = B_1 \times M^\prime$.

\section{Proof of Lemma \ref{lm3}}
In order to show the upper bound converges to $0$, we need the following lemma in our proof. Let $\delta = \frac 1 2 \min_{j \in \bm t}|\beta_{0,j}|$. 
\begin{lemma} \label{lm4}
	When $\bm k \nsupseteq \bm t$ and $\bm u = \bm k \cup \bm t$, $\frac{R_k^* - R_u^*}{\sigma_0^2} \sim \chi_{u-k}^2(\lambda),$ where $\lambda = \frac{\bm \beta_0^T\left(X_u^TP_*X_u\right)\bm \beta_0}{\sigma_0^2} > \frac{n\epsilon_n\delta^2}{\sigma_0^2}$ and $P_*$ is idempotent. 
\end{lemma}
\begin{proof}
	First write $X_u = \left[X_k,X_{t\cap k^c}\right]$ and $X_k = X_uW,$ where $W^T = \left(I_k,0_{k \times (u-k)}\right)^T.$ We let 
	\begin{align*}
	P_* = P_u - P_k &= X_u(X_u^TX_u)^{-1}X_u^T - X_k(X_k^TX_k)^{-1}X_k^T\\
	& = X_u(X_u^TX_u)^{-1}X_u^T - X_uW(W^TX_u^TX_uW)^{-1}W^TX_u^T.
	\end{align*} 
	Now further note that
	\begin{align*}
	P_*^2 =& \left[X_u(X_u^TX_u)^{-1}X_u^T - X_uW(W^TX_u^TX_uW)^{-1}W^TX_u^T\right]^2\\
	=& X_u(X_u^TX_u)^{-1}X_u^T - X_uW(W^TX_u^TX_uW)^{-1}W^TX_u^T \\
	&- X_uW(W^TX_u^TX_uW)^{-1}W^TX_u^T + X_uW(W^TX_u^TX_uW)^{-1}W^TX_u^T = P_*.
	\end{align*}
	Hence, $P_*$ is idempotent and $\frac{R_k^* - R_u^*}{\sigma_0^2} = \frac{\bm y_n^TP_*\bm y_n}{\sigma_0^2} \sim \chi^2_d(\lambda),$ where $d = rank(P_*) = rank(P_u) - rank(P_k) = u-k$ and $\lambda =  \frac{\bm \beta_0^T\left(X_u^TP_*X_u\right)\bm \beta_0}{\sigma_0^2}.$
	\\
	\noindent
	Next, note that 
	\begin{align*}
	P_*X_u = P_uX_u - P_kX_u = X_u - \left[X_k,P_kX_{t \cap k^c}\right] = \left[0_{n\times k}, \left(I_n -P_k\right)X_{t \cap k^c}\right].
	\end{align*}
	Therefore, $$X_u^TP_*X_u = \left(P_*X_u\right)^TP_*X_u  = \left[ \begin{matrix}
	0_{k \times k} & 0_{k \times t \cap k^c}\\
	0_{t \cap k^c \times k} & X_{t \cap k^c}^T\left(I_n -P_k\right)X_{t \cap k^c}
	\end{matrix}\right].$$
	Since $X_{t \cap k^c}^T\left(I_n -P_k\right)X_{t \cap k^c}$ is the Schur complement of $X_k^TX_k$, we have $$eig_1\left(X_u^TP_*X_u\right) > eig_1\left(X_k^TX_k\right) > n\epsilon_n.$$
	From the assumptions of our theorem, we have that the non-centrality parameter $$\lambda = \frac{\bm \beta_0^T\left(X_u^TP_*X_u\right)\bm \beta_0}{\sigma_0^2} > \frac{n\epsilon_n\delta^2}{\sigma_0^2}.$$
\end{proof} 
Let $Z$ be a standard normal distribution. By Lemma \ref{lm4}, it follows from the relation between noncentral chi-squared and normal distribution that, 
\begin{align} \label{finitesizeratio}
P\left(\exp\left[-\frac{\left(R_k^* - R_u^*\right)}{6\sigma_0^2}\right] > n^{-2rt}\right) \le &P\left(\frac{R_k^* - R_u^*}{\sigma_0^2} < 12rt\log n\right) \nonumber\\
< &P\left((Z-\sqrt{\lambda})^2 < 12rt\log n\right) \nonumber\\
< &P\left(Z > \sqrt{\lambda} - \sqrt{12rt\log n}\right) \nonumber\\
< &e^{- \frac{n\epsilon_n\delta^2}{4\sigma_0^2}} \rightarrow 0, \mbox{ as } n \rightarrow \infty.
\end{align}
Recall that $\bm u = \bm k \cup \bm t$ is a supset of $\bm t$. Therefore, $R_u^* \le R_k^*$. By (\ref{chisquaredPk}) and Assumption \ref{assumptiontruemodel}, for large enough $n$, say $n > N_1$ (not depending on $k$), we obtain,
\begin{equation} \label{Vbound}
\frac V{\epsilon_n^2} = \frac{1}{\epsilon_n^6}\hat \beta_{u}^T \hat \beta_{u}  \le \frac{1}{n\epsilon_n^7}\bm y^TP_u\bm y < \frac{1}{\epsilon_n^7}\log n < n^{\frac{1-\xi}{2r}}.
\end{equation}
Consider the first term of (\ref{mkt2}). When $\max\left\{\frac{4(r+1)}rt,\frac 4 {(1-\xi)^2}\right\} \le k < n^{\xi}$, we have $(r+1)t \le \frac 1 4 rk$ and $\sqrt{u-t} \le \sqrt k \le \frac{1-\xi}{2}k$. For large enough $n > N_1,$ using (\ref{chisquaredRk}), (\ref{chisquaredRk-t}) and Assumption \ref{assumptiontruemodel}, we have
\begin{align} \label{lm3proof1}
\begin{split}
& BA^k\left(\frac V{\epsilon_n^2}\right)^{rk}k^{k}n^{-(k-t)}\frac{\{R_t^* + 2\alpha_2\}^{\frac n 2 + rt + \alpha_1}}{\{R_k^* + 2\alpha_2\}^{\frac n 2 + rk + \alpha_1}}\\
\le & BA^k\left(\frac V{\epsilon_n^2}\right)^{rk}k^{k}n^{-(k-t)}\frac{\{R_t^*+2\alpha_2\}^{\frac n 2 + rt + \alpha_1}}{\{R_u^*+2\alpha_2\}^{\frac n 2 + rk + \alpha_1}}\\
\le &BA^kn^{\frac{1-\xi}{2} k}n^{\xi k}n^{-(k-t)}\left(1 + \frac{\frac{R_t^* - R_u^*}{\sigma_0^2}}{\frac{R_u^*}{\sigma_0^2}}\right)^{\frac n 2 + rt + \alpha_1}\frac1{{\left(R_u^*+2\alpha_2\right)}^{r(k - t)}}.
\end{split}
\end{align}
It follows from $1+x \le e^x$ and Assumption \ref{assumptiontruemodel} that
\begin{align*}
& BA^k\left(\frac V{\epsilon_n^2}\right)^{rk}k^{k}n^{-(k-t)}\frac{\{R_t^* + 2\alpha_2\}^{\frac n 2 + rt + \alpha_1}}{\{R_k^* + 2\alpha_2\}^{\frac n 2 + rk + \alpha_1}}\\
\le &K_1L_1^kn^{\frac{1-\xi}{2} k}n^{-(k-t)}n^{\xi k}\exp\left\{\frac{(u - t) + \sqrt{u - t}\log n}{n - u - \sqrt{n-u}\log n}\left(\frac n 2 + rt + \alpha_1\right)\right\}\\
&\times \left(\frac{\sigma_0^2}{n - u - \sqrt{n-u}\log n+2\alpha_2}\right)^{r(k - t)}\\
\le &K_1L_1^kn^{\frac{1-\xi}{2} k}n^{-(k-t)}n^{\xi k}n^{\sqrt{u-t}}n^{-r(k-t)}\\
\le &K_1{L_1}^kn^{-\frac 3 4 rk},
\end{align*}
where $K_1$ and $L_1$ are constants.\\
\noindent
When $k<\max\left\{\frac{4(r+1)}rt,\frac 4 {(1-\xi)^2}\right\}$, it follows by the boundedness of $k$, (\ref{chisquaredRt}), (\ref{chisquaredRk}), (\ref{chisquaredRk-t}) and (\ref{Vbound}) that the first part of (\ref{mkt2}) will be bounded by 
\begin{align} \label{lm3proof2}
\begin{split}
&BA^k\left(\frac V{\epsilon_n^2}\right)^{rk}k^{k}n^{-(k-t)}\frac{\{R_t^* + 2\alpha_2\}^{\frac n 2 + rt + \alpha_1}}{\{R_k^* + 2\alpha_2\}^{\frac n 2 + rk + \alpha_1}}\\
\le& BA^k\left(\frac V{\epsilon_n^2}\right)^{rk}k^{k}n^{-(k-t)}\frac{\{R_t^*+2\alpha_2\}^{\frac n 2 + rt + \alpha_1}}{\{R_u^*+2\alpha_2\}^{\frac n 2 + rk + \alpha_1}}\frac{\{R_u^*+2\alpha_2\}^{\frac n 2 + rk + \alpha_1}}{\{R_k^*+2\alpha_2\}^{\frac n 2 + rk + \alpha_1}}\\
\le & BA^k\left(\frac V{\epsilon_n^2}\right)^{rk}k^{k}n^{-(k-t)}\left(1 + \frac{\frac{R_t^* - R_u^*}{\sigma_0^2}}{\frac{R_u^*}{\sigma_0^2}}\right)^{\frac n 2 + rk + \alpha_1}{\left(R_t^*+2\alpha_2\right)}^{r(t - k)}\left(1 - \frac{\frac{R_k^* - R_u^*}{\sigma_0^2}}{\frac{R_k^* + 2\alpha_2}{\sigma_0^2}}\right)^{\frac n 2 + rk + \alpha_1}\\
\le&BA^k(\log n)^{2rk}n^{-(k-t)} \exp\left\{\frac{(u - t) + \sqrt{u-t}\log n}{n - u - \sqrt{n-u}\log n}\left(\frac n 2 + rk + \alpha_1\right)\right\} \\
&\times \left({n - t - \sqrt{n-t}\log n + 2\alpha_2}\right)^{r(t - k)} \left(1 - \frac{\frac{R_k^* - R_u^*}{\sigma_0^2}}{\frac{R_k^* + 2\alpha_2}{\sigma_0^2}}\right)^{\frac n 2 + rk + \alpha_1}.
\end{split}
\end{align}
By $1+x \le e^x$ and $1-x \le e^{-x}$, we get
\begin{align*}
& BA^k\left(\frac V{\epsilon_n^2}\right)^{rk}k^{k}n^{-(k-t)}\frac{\{R_t^* + 2\alpha_2\}^{\frac n 2 + rt + \alpha_1}}{\{R_k^* + 2\alpha_2\}^{\frac n 2 + rk + \alpha_1}}\\
\le &BA^k(\log n)^{2rk}n^{-(k-t)}n^{\sqrt{u-t}-r(k-t)}\\
&\times \exp\left\{-\frac{\frac{R_k^* - R_u^*}{\sigma_0^2}}{n - k + \sqrt{n-k}\log n + 2\alpha_2}\left(\frac n 2 + rk + \alpha_1\right)\right\}\\
\le & K_2L_2^k(\log n)^{2rk}n^{k - (r+1) (k-t)}\exp\left\{-\frac{R_k^* - R_u^*}{6\sigma_0^2}\right\}\\
< &K_2 L_2^k(\log n)^{2rk}n^{-rk}n^{(r+1)t}n^{-2rt}\\
< & K_2 L_2^k n^{-\frac 3 4 rk},
\end{align*}
for large enough constants $K_2$, $L_2$, since $(\log n)^{2rk}  < n^{\frac14 rk}$ for large enough $n$, say $n > N_3$ (not depending on $k$), and $\sqrt{u - t} \le \sqrt k \le k$.\\
Now consider the second term of (\ref{mkt2}). When $\max\left\{\frac {\frac 5 4 r + 1}{1-\xi}t,\frac 1 {(1-\xi)^2}\right\} \le k \le n^{\xi}$, we have $\sqrt{u-t} \le \sqrt{k} \le (1-\xi)k$ and $(\frac 5 4 r + 1)t \le (1-\xi)rk$. By (\ref{chisquaredRk}), (\ref{chisquaredRk-t}) and Assumption \ref{assumptionmodelsize}, for large enough $n,$ say $n>N_4$ (not depending on $k$), we have 
\begin{align} \label{lm3proof3}
\begin{split}
&BA^kk^{(r+1)k}n^{-(r+1)(k-t)-\frac 3 4 rk - rt}\frac{\{R_t^* + 2\alpha_2\}^{\frac n 2 + rt + \alpha_1}}{\{R_k^* + 2\alpha_2\}^{\frac n 2 + \alpha_1}}\\
\le & BA^kk^{(r+1)k}n^{-(r+1)(k-t)-\frac 3 4 rk - rt}\frac{\{R_t^* + 2\alpha_2\}^{\frac n 2 + rt + \alpha_1}}{\{R_u^* + 2\alpha_2\}^{\frac n 2 + \alpha_1}}\\
\le & BA^kn^{\xi(r+1)k}n^{-(r+1)(k-t)-\frac 3 4 rk - rt}\left(1 + \frac{\frac{R_t^* - R_u^*}{\sigma_0^2}}{\frac{R_u^*}{\sigma_0^2}}\right)^{\frac n 2 + rt + \alpha_1}{\left(R_u^*+2\alpha_2\right)}^{rt}.
\end{split}
\end{align}
Similar to (\ref{lm3proof1}), we obtain,
\begin{align*}
&BA^kk^{(r+1)k}n^{-(r+1)(k-t)-\frac 3 4 rk - rt}\frac{\{R_t^* + 2\alpha_2\}^{\frac n 2 + rt + \alpha_1}}{\{R_k^* + 2\alpha_2\}^{\frac n 2 + \alpha_1}}\\
\le & BA^kn^{\xi(r+1)k}n^{-(r+1)(k-t)-\frac 3 4 rk - rt}\left({n - u - \sqrt{n-u}\log n + 2\alpha_2}\right)^{rt}\\
&\times \exp\left\{\frac{(u - t) + \sqrt{u-t}\log n}{n - u - \sqrt{n-u}\log n}\left(\frac n 2 + rt + \alpha_1\right)\right\} \\
\le & K_3L_3^kn^{\xi(r+1)k}n^{-(r+1)(k-t)-\frac 3 4 rk - rt}n^{\sqrt{u-t}+rt}\\
\le & K_3L_3^kn^{-\frac 3 4 rk},
\end{align*}
where $K_3$ and $L_3$ are large enough constants.\\
When $k < \max\left\{\frac {\frac 5 4 r + 1}{1-\xi}t,\frac 1 {(1-\xi)^2}\right\},$ using (\ref{finitesizeratio}) and argument similar to (\ref{lm3proof2}), we have
\begin{align*}
&BA^kk^{(r+1)k}n^{-(r+1)(k-t)-\frac 3 4 rk - rt}\frac{\{R_t^* + 2\alpha_2\}^{\frac n 2 + rt + \alpha_1}}{\{R_k^* + 2\alpha_2\}^{\frac n 2 + \alpha_1}}\\
\le &BA^kn^{-(r+1)(k-t)-\frac 3 4 rk - rt}\frac{\{R_t^*+2\alpha_2\}^{\frac n 2 + rt + \alpha_1}}{\{R_u^*+2\alpha_2\}^{\frac n 2 + \alpha_1}}\frac{\{R_u^*+2\alpha_2\}^{\frac n 2 + \alpha_1}}{\{R_k^*+2\alpha_2\}^{\frac n 2 + \alpha_1}}\\
\le &BA^kn^{-(r+1)(k-t)-\frac 3 4 rk - rt}\left(1 + \frac{\frac{R_t^* - R_u^*}{\sigma_0^2}}{\frac{R_u^*}{\sigma_0^2}}\right)^{\frac n 2  + \alpha_1}{\left(R_t^*+2\alpha_2\right)}^{rt}\\
&\times \left(1 - \frac{\frac{R_k^* - R_u^*}{\sigma_0^2}}{\frac{R_k^* + 2\alpha_2}{\sigma_0^2}}\right)^{\frac n 2 + \alpha_1}.
\end{align*} 
By $1+x \le e^x$, we get,
\begin{align*}
&BA^kk^{(r+1)k}n^{-(r+1)(k-t)-\frac 3 4 rk - rt}\frac{\{R_t^* + 2\alpha_2\}^{\frac n 2 + rt + \alpha_1}}{\{R_k^* + 2\alpha_2\}^{\frac n 2 + \alpha_1}}\\
\le &BA^kn^{-(r+1)(k-t)-\frac 3 4 rk - rt}\exp\left\{\frac{(u - t) + \sqrt{u-t}\log n}{n - u - \sqrt{n-u}\log n}\left(\frac n 2 + \alpha_1\right)\right\} \\
&\times \left({n - t - \sqrt{n-t}\log n + 2\alpha_2}\right)^{rt}\exp\left\{-\frac{\frac{R_k^* - R_u^*}{\sigma_0^2}}{n - k + \sqrt{n-k}\log n + 2\alpha_2}\left(\frac n 2 + \alpha_1\right)\right\}\\
\le & K_4L_4^kn^{-(r+1)(k-t)-\frac 3 4 rk - rt}n^{\sqrt{u-t}}n^{rt}n^{-2rt}\\
\le & K_4L_4^kn^{-\frac 3 4 rk}.
\end{align*}
where $K_4$ and $L_4$ are large enough constants, for large enough $n$, say $n >  N_6$ (not depending on $k$).

Combining all cases and letting $N^\prime = \max\{N, N_1,N_2,N_3,N_4,N_5,N_6\}$ (not depending on $k$), we get that
\begin{align}
\frac{m_{\bm k}(\bm y_n)}{m_{\bm t}(\bm y_n)} < K^{\prime}(L^{\prime})^kn^{-\frac 3 4 rk},
\end{align}	
where $K^\prime = \max\{K_1,K_2,K_3,K_4\}$ and $L^\prime = \max\{L_1,L_2,L_3,L_4\}$ for all $n > N^\prime$.

\section{Proof of Lemma \ref{lm5}}
Similar to (\ref{Vbound}), for large enough $n$, say $n > N_1^\prime$ (not depending on $k$), we have $V < n^{\frac{1-\xi}{2r}}$. Next,
consider the first term in (\ref{mkt2}). When $\max\left\{\frac{2(1+\xi)+4r}r t, t + \frac 4 {(1-\xi)^2}\right\} \le k \le n^\xi,$ we obtain $\frac 1 4 rk \ge \frac{1+\xi}2t + rt$ and $\sqrt{k - t} \le \frac{1-\xi}2 (k - t)$. For large enough $n$, say $n > N_2^\prime$ (not depending on $k$), when $\bm k \supset \bm t,$ using (\ref{chisquaredRk}), (\ref{chisquaredRk-t}) and Assumption \ref{assumptionmodelsize}, we get
\begin{align} \label{lm5proof1}
\begin{split}
& BA^k\left(\frac V{\epsilon_n^2}\right)^{rk}k^{k}n^{-(k-t)}\frac{\{R_t^* + 2\alpha_2\}^{\frac n 2 + rt + \alpha_1}}{\{R_k^* + 2\alpha_2\}^{\frac n 2 + rk + \alpha_1}}\\
\le &BA^kn^{\frac{1-\xi}{2} k}n^{\xi k}n^{-(k-t)}\left(1 + \frac{\frac{R_t^* - R_k^*}{\sigma_0^2}}{\frac{R_k^*}{\sigma_0^2}}\right)^{\frac n 2 + rt + \alpha_1}\frac1{{\left(R_k^*+2\alpha_2\right)}^{r(k - t)}}\\
\le &S_1T_1^kn^{\frac{1-\xi}{2} k}n^{-(k-t)}n^{\xi k}\exp\left\{\frac{(k - t) + \sqrt{k - t}\log n}{n - k - \sqrt{n-k}\log n}\left(\frac n 2 + rt + \alpha_1\right)\right\}\\
&\times \left(\frac{\sigma_0^2}{n - k - \sqrt{n-u}\log n+2\alpha_2}\right)^{r(k - t)}\\
\le &S_1T_1^kn^{\frac{1-\xi}{2} k}n^{-(k-t)}n^{\xi k}n^{\sqrt{k-t}}n^{-r(k-t)}\\
\le &S_1{T_1}^kn^{-\frac 3 4 rk},
\end{split}
\end{align}
where $S_1$ and $T_1$ are large enough constants.\\
When $t < k < \max\left\{\frac{2(1+\xi)+4r}r t, t + \frac 4 {(1-\xi)^2}\right\},$ by the boundedness of $k$, (\ref{chisquaredRt}), (\ref{chisquaredRk}) and (\ref{chisquaredRk-t}), the first part of (\ref{mkt2}) will be bounded by
\begin{align} \label{lm5proof2}
\begin{split}
& BA^k\left(\frac V{\epsilon_n^2}\right)^{rk}k^{k}n^{-(k-t)}\frac{\{R_t^* + 2\alpha_2\}^{\frac n 2 + rt + \alpha_1}}{\{R_k^* + 2\alpha_2\}^{\frac n 2 + rk + \alpha_1}}\\
\le & BA^k\left(\frac V{\epsilon_n^2}\right)^{rk}k^{k}n^{-(k-t)}\left(1 + \frac{\frac{R_t^* - R_k^*}{\sigma_0^2}}{\frac{R_u^*}{\sigma_0^2}}\right)^{\frac n 2 + rk + \alpha_1}{\left(R_t^*+2\alpha_2\right)}^{r(t - k)}\\
\le&BA^k(\log n)^{2rk}n^{-(k-t)} \exp\left\{\frac{(k - t) + \sqrt{k-t}\log n}{n - k - \sqrt{n-k}\log n}\left(\frac n 2 + rk + \alpha_1\right)\right\} \\
&\times \left({n - t - \sqrt{n-t}\log n + 2\alpha_2}\right)^{r(t - k)} \\
\le& S_2T_2^{k-t}(\log n)^{2rk}n^{-(k-t)}n^{\sqrt{k-t}}n^{-r(k-t)}\\
\le &S_2T_2^{k-t}n^{-\frac 3 4 r(k-t)},
\end{split}
\end{align}
for large enough constants $S_2$, $T_2$, since $\sqrt{k-t} \le k-t$ and $(\log n)^{2rk}  < n^{\frac 1 4 r} \le  n^{\frac14 r(k-t)}$ for large enough $n$, say $n > N_3^\prime$ (not depending on $k$).
\\
\noindent
Now for the second part of (\ref{mkt2}), when $\max\left\{\frac {16} {r^2} + t, \frac{1+\frac 3 4 r}{ r}t\right\} \le k \le n^\xi,$ we have $\sqrt{k - t} \le \frac r 4 (k-t)$ and $(1+\frac 3 4 r)t \le \frac 1 2 rk.$ Again using (\ref{chisquaredRk}), (\ref{chisquaredRk-t}) and Assumption \ref{assumptionmodelsize}, for large enough $n$, say $n > N_4^\prime$ (not depending on $k$), we get
\begin{align} \label{lm5proof3}
\begin{split}
&BA^kk^{(r+1)k}n^{-(r+1)(k-t)-\frac 3 4 rk - rt}\frac{\{R_t^* + 2\alpha_2\}^{\frac n 2 + rt + \alpha_1}}{\{R_k^* + 2\alpha_2\}^{\frac n 2 + \alpha_1}}\\
\le & BA^kn^{\xi(r+1)k}n^{-(r+1)(k-t)-\frac 3 4 rk - rt}\left(1 + \frac{\frac{R_t^* - R_k^*}{\sigma_0^2}}{\frac{R_k^*}{\sigma_0^2}}\right)^{\frac n 2 + rt + \alpha_1}{\left(R_k^*+2\alpha_2\right)}^{rt}\\
\le & BA^kn^{\xi(r+1)k}n^{-(r+1)(k-t)-\frac 3 4 rk - rt}\exp\left\{\frac{(k - t) + \sqrt{k-t}\log n}{n - k - \sqrt{n-k}\log n}\left(\frac n 2 + rt + \alpha_1\right)\right\} \\
&\times \left({n - k - \sqrt{n-k}\log n + 2\alpha_2}\right)^{rt}\\
\le & S_3T_3^kn^{\xi(r+1)k}n^{-(r+1)(k-t)-\frac 3 4 rk - rt}n^{\sqrt{k-t}}n^{rt}\\
\le & S_3T_3^kn^{-(1-\xi)(r + 1)k},	
\end{split}
\end{align}
for large enough constants $S_3$ and $T_3$.\\
\noindent
When $t < k < \max\left\{\frac {16} {r^2} + t, \frac{1+\frac 3 4 r}{\frac 1 2 r}t\right\}, $ following the exact arguments in (\ref{lm5proof2}), using the boundedness of $k$ and $\sqrt{k - t} \le k - t$, it follows that for large enough $n$, say $n > N_5^\prime$ (not depending on $k$), the second part of (\ref{mkt2}) is bounded by $S_4T_4^kn^{-r(k-t)},$ for large enough constants $S_4$ and $T_4$.

Combining all cases and letting $N^{\prime\prime} = \max\{N, N_1^\prime, N_2^\prime,N_3^\prime,N_4^\prime, N_5^\prime\}$ (not depending on $k$), we have
\begin{align}
\frac{m_{\bm k}(\bm y_n)}{m_{\bm t}(\bm y_n)} < S^{\prime}(T^{\prime})^{k-t}n^{-\min\left\{\frac 3 4, 1-\xi\right\} r(k-t)},
\end{align}	
for $n > N^{\prime\prime}$, where $S^\prime = \max\{S_1,S_2,S_3,S_4\}$ and $T^\prime = \max\{T_1,T_2,T_3,T_4\}$.

\section{Proof of Lemma \ref{lm6} to \ref{lm8}}
It follows from Lemma \ref{lm1} that, under the complexity prior, for large enough $n > N$ (not depending on $k$),
\begin{align} \label{pr_complexity}
\frac{\pi(\bm k |\bm y_n)}{\pi(\bm t | \bm y_n)} &\le c_1^{t-k}p^{c_2(t-k)} BA^k\left(\frac V{\epsilon_n^2}\right)^{rk}k^{k}n^{-(k-t)}\frac{\{R_t^* + 2\alpha_2\}^{\frac n 2 + rt + \alpha_1}}{\{R_k^* + 2\alpha_2\}^{\frac n 2 + rk + \alpha_1}} \nonumber\\
&+ c_1^{t-k}p^{c_2(t-k)}BA^kk^{(r+1)k}n^{-(r+1)(k-t)-\frac 3 4 rk - rt} \frac{\{R_t^* + 2\alpha_2\}^{\frac n 2 + rt + \alpha_1}}{\{R_k^* + 2\alpha_2\}^{\frac n 2 + \alpha_1}},
\end{align}
where $A = A_1 \times M^\prime$ and $B = B_1 \times M^\prime$, $A_1$, $B_1$, $m^\prime$ are large enough constants.\\
When $\bm k \subset \bm t$, note that $\bm u = \bm k \cup \bm t = \bm t$. Therefore, for large enough $n$, say $n > N_1$ (not depending on $k$), $\frac V {\epsilon_n^2} < (\log n)^2.$ By following the similar arguments in (\ref{finitesizeratio}), we obtain
\begin{align} \label{finiteR_ratio}
P\left(\exp\left[-\frac{\left(R_k^* - R_t^*\right)}{6\sigma_0^2}\right] > p^{-3(c_2+r)t}\right) \le &P\left(\frac{R_k^* - R_u^*}{\sigma_0^2} < 18(c_2+r)t\log p\right) \nonumber\\
< &P\left((Z-\sqrt{\lambda})^2 < 18(c_2+r)t\log p\right) \nonumber\\
< &P\left(Z > \sqrt{\lambda} - \sqrt{18(c_2+r)t\log p}\right) \nonumber\\
< &e^{- \frac{n\epsilon_n\delta^2}{4\sigma_0^2}} \rightarrow 0, \mbox{ as } n \rightarrow \infty.
\end{align} 
Next, consider the first term in (\ref{pr_complexity}). With probability tending to 1, it follows from $1-x \le e^{-x}$, and (\ref{finiteR_ratio}) that, for large enough $n > N_1^{\prime\prime}$ (not depending on $k$) and large enough constant $M_1$, we have
\begin{align} \label{thm4proof1}
&c_1^{t-k}p^{c_2(t-k)} BA^k\left(\frac V{\epsilon_n^2}\right)^{rk}k^{k}n^{-(k-t)}\frac{\{R_t^* + 2\alpha_2\}^{\frac n 2 + rt + \alpha_1}}{\{R_k^* + 2\alpha_2\}^{\frac n 2 + rk + \alpha_1}} \nonumber\\
\le & M_1p^{c_2(t-k)}\left(\frac V{\epsilon_n^2}\right)^{rk}k^{k}n^{-(k-t)}{\left(R_t^*+2\alpha_2\right)}^{r(t - k)}\left(1 - \frac{\frac{R_k^* - R_t^*}{\sigma_0^2}}{\frac{R_k^* + 2\alpha_2}{\sigma_0^2}}\right)^{\frac n 2 + rk + \alpha_1} \nonumber \\
\le& M_1p^{c_2(t-k)}(\log n)^{2rk}n^{-(k-t)} \left({n - t - \sqrt{n-t}\log n + 2\alpha_2}\right)^{r(t - k)} \nonumber \\ & \times \exp\left\{-\frac{\frac{R_k^* - R_t^*}{\sigma_0^2}}{n - k + \sqrt{n-k}\log n + 2\alpha_2}\left(\frac n 2 + rk + \alpha_1\right)\right\}\nonumber \\
\le & M_1p^{c_2t}n^{(r+1)t}\exp\left\{-\frac{R_k^* - R_t^*}{6\sigma_0^2}\right\}\nonumber \\
\le &M_1p^{-2c_2t} \rightarrow 0, \quad \mbox{as } n \rightarrow \infty.
\end{align}
Now consider the second term in (\ref{pr_complexity}). When $\bm k \subset \bm t$, for large enough $n > N_1^{\prime\prime}$ (not depending on $k$), we obtain, 
\begin{align} \label{thm4proof2}
&c_1^{t-k}p^{c_2(t-k)}BA^kk^{(r+1)k}n^{-(r+1)(k-t)-\frac 3 4 rk - rt} \frac{\{R_t^* + 2\alpha_2\}^{\frac n 2 + rt + \alpha_1}}{\{R_k^* + 2\alpha_2\}^{\frac n 2 + \alpha_1}} \nonumber \\
\le & M_1p^{c_2t} n^t (R_t^* + 2\alpha_2)^{rt}\left(1 - \frac{\frac{R_k^* - R_t^*}{\sigma_0^2}}{\frac{R_k^* + 2\alpha_2}{\sigma_0^2}}\right)^{\frac n 2 + \alpha_1} \nonumber \\
\le & M_1p^{c_2t}n^{(r+1)t}\exp\left\{-\frac{R_k^* - R_t^*}{6\sigma_0^2}\right\}\le M_1 p^{-2c_2t} \rightarrow 0, \quad \mbox{as } n \rightarrow \infty.
\end{align}
Therefore, when $\bm k \subset \bm t$, for large enough $n > N_1^{\prime\prime}$ (not depending on $k$), with probability tending to 1,
\begin{align} \label{thm4proof3}
\frac{\pi(\bm k |\bm y_n)}{\pi(\bm t | \bm y_n)}  \le 2M_1 p^{-2c_2t} \rightarrow 0, \quad \mbox{as } n \rightarrow \infty.
\end{align}
When $\bm k \supset \bm t$, it follows from Lemma \ref{lm5} and (\ref{pr_complexity}) that, for large enough $n > N^{\prime\prime}$ (not depending on $k$),
\begin{align}  \label{thm4proof4}
\frac{\pi(\bm k |\bm y_n)}{\pi(\bm t | \bm y_n)} \le c_1^{-(k-t)}p^{-c_2(k-t)} \rightarrow 0, \quad \mbox{as } n \rightarrow \infty.
\end{align}
Next, when $\bm k \nsubseteq \bm t,$ $\bm k \nsupseteq \bm t$ and $\bm k \neq \bm t$, denote $\bm u = \bm k \cup \bm t$. First we discuss the scenario when $k \le t$. It follows from (\ref{pr_complexity}) and $R_u^* \le R_k^*$ that, for large enough $n > N$ (not depending on $k$), 
\begin{align}
\frac{\pi(\bm k |\bm y_n)}{\pi(\bm t | \bm y_n)} &\le c_1^{t-k}p^{c_2(t-k)} BA^k\left(\frac V{\epsilon_n^2}\right)^{rk}k^{k}n^{-(k-t)}\frac{\{R_t^* + 2\alpha_2\}^{\frac n 2 + rt + \alpha_1}}{\{R_u^* + 2\alpha_2\}^{\frac n 2 + rk + \alpha_1}} \nonumber\\
&+ c_1^{t-k}p^{c_2(t-k)}BA^kk^{(r+1)k}n^{-(r+1)(k-t)-\frac 3 4 rk - rt} \frac{\{R_t^* + 2\alpha_2\}^{\frac n 2 + rt + \alpha_1}}{\{R_u^* + 2\alpha_2\}^{\frac n 2 + \alpha_1}}.
\end{align}
By following the exact same arguments that lead up to (\ref{thm4proof3}), we can show for large enough $n > N_2^{\prime\prime}$ (not depending on $k$) and large enough constant $M_2 > 2M_1$, 
\begin{align}  \label{thm4proof5}
\frac{\pi(\bm k |\bm y_n)}{\pi(\bm t | \bm y_n)}  \le M_2p^{-2c_2t}  \rightarrow 0, \quad \mbox{as } n \rightarrow \infty.
\end{align}
Now if $k > t$, by (\ref{pr_complexity}), for large enough $n > N$, we obtain
\begin{align}
&\frac{\pi(\bm k |\bm y_n)}{\pi(\bm t | \bm y_n)} \nonumber\\
\le& c_1^{t-k}p^{c_2(t-k)} BA^k\left(\frac V{\epsilon_n^2}\right)^{rk}k^{k}n^{-(k-t)}\frac{\{R_t^* + 2\alpha_2\}^{\frac n 2 + rt + \alpha_1}}{\{R_u^* + 2\alpha_2\}^{\frac n 2 + rk + \alpha_1}}\frac{\{R_u^* + 2\alpha_2\}^{\frac n 2 + rk + \alpha_1}}{\{R_k^* + 2\alpha_2\}^{\frac n 2 + rk + \alpha_1}} \nonumber\\
&+ c_1^{t-k}p^{c_2(t-k)}BA^kk^{(r+1)k}n^{-(r+1)(k-t)-\frac 3 4 rk - rt} \frac{\{R_t^* + 2\alpha_2\}^{\frac n 2 + rt + \alpha_1}}{\{R_u^* + 2\alpha_2\}^{\frac n 2 + \alpha_1}}\frac{\{R_u^* + 2\alpha_2\}^{\frac n 2 + \alpha_1}}{\{R_k^* + 2\alpha_2\}^{\frac n 2 + \alpha_1}}.
\end{align}
Following the similar arguments in (\ref{lm3proof2}), we can show for large enough $n > N_3^{\prime \prime}$ (not depending on $k$), with probability tending to 1, 
\begin{align}  \label{thm4proof6}
\frac{\pi(\bm k |\bm y_n)}{\pi(\bm t | \bm y_n)} \le c_3^{-(k-t)}p^{-c_2k} \rightarrow 0, \quad \mbox{as } n \rightarrow \infty.
\end{align}  
for some constant $c_3 > 0$.  
\end{document}